\documentclass[a4paper,12pt]{article}

\usepackage[margin=3cm,footskip=1cm]{geometry}

\usepackage[T1]{fontenc}

\usepackage{soul}

\usepackage{authblk}

\usepackage[title]{appendix}

\usepackage[above,below,section]{placeins}

\usepackage{amsmath,amssymb,natbib}
\usepackage{graphicx,xcolor}
\usepackage{enumerate}
\usepackage{color}
\usepackage{multirow}
\usepackage{tablefootnote}
\usepackage{threeparttable}
\usepackage{pdflscape}

\usepackage{algorithmicx}
\usepackage[ruled]{algorithm}
\usepackage{algpseudocode}
\usepackage{color}
\usepackage{array}
\newcolumntype{H}{>{\setbox0=\hbox\bgroup}c<{\egroup}@{}}

\def \R { \mathbb{R} }

\newcommand{\dd}{\mathrm{d}}
\newcommand{\bbeta}{\boldsymbol{\beta}}

\newcommand{\bz}{\mathbf{z}}

\renewcommand{\P}{\mathrm P}
\newcommand{\gm}{\gamma}
\newcommand{\ta}{\theta}

\newcommand{\bDelta}{\mathbf \Delta}
\newcommand{\bSigma}{\mathbf \Sigma}

\newcommand{\bZ}{\mathbf{Z}}
\newcommand{\bC}{\mathbf{C}}
\newcommand{\bA}{\mathbf{A}}
\newcommand{\bX}{\mathbf{X}}
\newcommand{\bx}{\mathbf{x}}

\newcommand{\bS}{\mathbf{S}}
\newcommand{\bic}{\mathrm{BIC}}

\newcommand{\cbic}{\mathrm{CBIC}}
\renewcommand{\det}{\mathrm{det}}

\newcommand{\EE}{{\mathbb E}}
\newcommand{\Var}{{\mathbb{V}\mathrm{ar}}}

\DeclareMathOperator*{\argmax}{arg\,max}

\definecolor{darkturquoise}{rgb}{0.0, 0.81, 0.82}

\title{Information criteria for inhomogeneous spatial point processes}

\author[1]{Achmad Choiruddin}
\affil[1]{Department of Statistics, Institut Teknologi Sepuluh Nopember (ITS), Indonesia}

\author[2]{Jean-Fran{\c c}ois Coeurjolly}
\affil[2]{Department of Mathematics, Universit\'e du Qu\'ebec \`a Montr\'eal (UQAM), Canada}

\author[3]{Rasmus Waagepetersen}
\affil[3]{Department of Mathematical Sciences, Aalborg University, Denmark}

\usepackage{amsthm}
\newtheorem{theorem}{Theorem}

\newtheorem{proposition}{Proposition}

\newtheorem{remark}{Remark}
\newcommand{\NN}{\mathbb N}
\newtheorem{example}{Example}

\begin{document}
	
	\maketitle
	
	\begin{abstract}
          The theoretical foundation for a number of model selection criteria is established in the context of inhomogeneous point processes and under various asymptotic settings: infill, increasing domain, and combinations of these. For inhomogeneous Poisson processes we consider Akaike's information criterion and the Bayesian information criterion, and in particular we identify the point process analogue of `sample size' needed for the Bayesian information criterion. Considering general inhomogeneous point processes we derive new composite likelihood and composite Bayesian information criteria for selecting a regression model for the intensity function. The proposed model selection criteria are evaluated using simulations of Poisson processes and cluster point processes.
		\end{abstract}

	\section{Introduction}

        Fitting a regression model to the intensity function of a point process is one of the most fundamental tasks in statistical analysis of point pattern data , see e.g. \citet{moeller:waagepetersen:17} or \citet{coeurjolly:lavancier:19} for a recent review of this problem. If the data in question can be viewed as a realization of a Poisson process, regression parameters are usually estimated by maximum likelihood. If the point process is not Poisson, the likelihood function is often computationally intractable. In such cases the Poisson likelihood function can still be used as a composite likelihood function for estimating regression parameters. This is e.g.\ the approach underlying the popular  {\texttt{spatstat R}} package \citep{baddeley:rubak:turner:15} procedure \texttt{kppm}.
        
 Considering regression models, model selection is often a pertinent task. In case of a Poisson process, the Akaike information criterion (AIC) \citep{akaike:73}  seems an obvious approach (and implemented in the function \texttt{logLik.ppm} of \texttt{spatstat}) since in this case the likelihood function is available. If the assumption of a Poisson process is not tenable, generalization of the AIC to a composite likelihood information criterion (CIC) \citep{varin:vidoni:05} is relevant. Yet another alternative is the Bayesian information criterion (BIC) \citep{schwarz:78}. The classical framework for the information criteria mentioned typically involves a sample of independent observations where the sample size plays a crucial role for the BIC. However, in the case of a unique realization of a point process, it is not obvious how to define sample size. In some sense, it is one, but this choice is obviously not useful for asymptotic justifications of information criteria. Instead, sample size must be linked to properties of the observation window or the point process intensity function. 
 Several proposals for defining `sample size' have been considered in the literature. \citet{choiruddin:coeurjolly:letue:18} use the size of the observation window while \citet{thurman:fu:guan:15} consider the number of observed points. \citet{jeffrey:etal:18} use the sum of the number of data points and the number of dummy points used in a numerical approximation of the likelihood \citep{berman:turner:92}.

In this paper we first establish the theoretical foundation for AIC and CIC in the context of intensity function model selection for a point process. This includes asymptotic results for estimates of the `least false parameter value', see e.g. \citet[][Section 2.2]{claeskens:hjort:08}.
Next we derive the BIC in case of  a Poisson process and we thus identify what is the meaning of sample size in this context. We also consider the generalization of BIC to composite likelihood BIC (CBIC) \citep{gao:song:10} using a concept of effective degrees of freedom derived for the CIC. Our asymptotic developments are established under an original setting which embraces both infill asymptotic (the number of points in a fixed domain increases) and increasing domain asymptotics (the volume of the observation window tends to infinity) which are often considered in the literature. 

The rest of the article is organized as follows. The problem of selecting a model for the intensity function is specified  in Section~\ref{sec:modelselection}. In Section~\ref{sec:estInt} we discuss asymptotic results for intensity function regression parameter estimators under a 'double' asymptotic framework. We derive the AIC and CIC for spatial point processes in Section~\ref{sec:cic} and develop the BIC and CBIC in Section~\ref{sec:bic}. The different model selection criteria are compared in a simulation study in Section~\ref{sec:sim}. Section~\ref{sec:concl} gives some concluding remarks. Proofs are given in the  Appendices~\ref{sec:mlestar}-\ref{sec:wg09}.

\section{Intensity model selection}\label{sec:modelselection}

A spatial point process $\bX$ defined on $\R^d$ is a locally finite random subset of $\R^d$. If for bounded $B  \subset S$ we denote by $N(B)$ the cardinality of $\bX \cap B$, locally finite means that $N(B)$ is a finite integer almost surely.
For a bounded domain $W \subset S$, $|W|$ denotes the volume of $W$. The intensity function $\lambda$ and the pair correlation function $g$ of $\bX$ are defined (if they exist) by the equations
\begin{align*}
\EE N(A) &= \int_A \lambda(u) \dd u \\
 \EE \left\{ N(A)N(B) \right\}&= \int_A \lambda(u) \dd u + \int_A \int_B \lambda(u) \lambda(v) g(u,v) \dd u \dd v  	
\end{align*}
for any bounded $A,B \subset \R^d$.

If the counts $N(A)$ are Poisson distributed, $\bX$ is said to be a Poisson process. In this case counts $N(B_1),\dots,N(B_m)$ are independent whenever the subsets $B_1,\ldots,B_m$ are disjoint and the pair correlation function is identically equal to one. 
For our asymptotic considerations, we assume that a sequence of spatial point processes $\bX_n$ is observed within a sequence of bounded windows $W_n \subset \R^d$, $n=1,2,\ldots$. We denote by $\lambda_n$ and $g_n$ the intensity and pair correlation function of $\bX_n$. With an abuse of notation, we denote for any $n \ge 1$ expectation and variance under the sampling distribution of $\bX_n$ by $\EE$ and $\Var$.

For modelling the intensity function we assume that $p \ge 1$ covariates $z_1,\ldots,z_p$ are available where for each $i=1,\ldots,p$, $z_i$ is a locally integrable function on $\R^d$. Let $I_l$, $l=1,\ldots,2^p,$ denote the subsets of $\{1,\dots,p\}$ and let $p_l=|I_l|+1$ where $|I_l|$ is the cardinality of $I_l$. We consider models for $\lambda_n$ specified in terms of the $2^p$ subsets of the covariates. For each $l=1,\ldots,2^p$ and $n \ge 1$ we  define the log-linear model
	\begin{equation}\label{eq:loglinear} 
	\rho_{l,n}\{u;\bz_l(u);\bbeta_l\}=\theta_n \exp\{ \bbeta_l^\top \bz_l(u)\} \end{equation}
	where $\bbeta_l=\left \{\beta_0,(\beta_j)_{j\in I_l} \right\} \in
	\R^{p_l}$ and $\bz_l(u) = \big[ 1, \{z_j(u)\}_{j\in I_l} \big]$.  The quantity $\theta_n$ should not be regarded as a parameter to be estimated. For $n \ge 1$, $\theta_n$ could e.g.\ represent a timespan over which $\bX_n$ is observed.  In the following, with an abuse of notation, we just write $\rho_{n}(u;\bbeta_l)$ for $\rho_{l,n}\{u;\bz_l(u);\bbeta_l\}$ and similarly for related quantities.

The problem we consider is to select among the $2^p$ intensity models $\mathcal M_l$, $l=1,\dots,2^p$, given by 
	\[
	\mathcal M_l = \{  \rho_{n}(\cdot ;\bbeta_l ) \mid \bbeta_l = \left\{\beta_0,(\beta_j)_{j\in I_l}\right\} \in \R^{p_l} \}, \quad l=1,\dots,2^p.
	\] 
	The distinction between $\beta_0$ and the other parameters is necessary because our objective is to select among $2^p$ different models which all contain an intercept term. Note that the true intensity function $\lambda_n$ does not necessarily correspond to any of the suggested models ${\cal M}_l$, $l=1,\ldots,2^p$.

 \section{Estimation of the intensity function} \label{sec:estInt}

In this section we discuss estimation of the intensity function
  using a Poisson likelihood function 
  and associated asymptotic results.
\subsection{The Poisson likelihood function}\label{sec:correct}

    The density of a Poisson point process with intensity ${\rho_n}(\cdot;\bbeta_l)$ and observed in $W_n$ is given by (see e.g. \cite{moeller:waagepetersen:04})
	\begin{equation}\label{eq:fn}
	p_{n}(\bx ; \bbeta_l) = \left\{ \prod_{u \in {\bx}} {\rho_{n}}(u;\bbeta_l) \right\} \exp\left\{ |W_n|-\int_{W_n} {\rho_{n}}(u;\bbeta_l) \dd u\right\} 
	\end{equation} 
	for locally finite point configurations {$\bx \subset W_n$}. 
We emphasize that in the following we assume neither that $\bX_n$ introduced in the previous section is a Poisson process nor that ${\rho_{n}}(\cdot;\bbeta_l)$ coincides with the intensity function $\lambda_n$ of $\bX_n$.

Combining \eqref{eq:loglinear} and \eqref{eq:fn}, up to a constant, the $\log$ of \eqref{eq:fn} evaluated at $\bX_n$ becomes
	\begin{equation}\label{eq:loglikelihood}
	\ell_{n}(\bbeta_l) = \sum_{u \in {\bX_n}} \bbeta_l^\top \bz_l(u) - \theta_n \int_{W_n} \exp\{ \bbeta_l^\top \bz_l(u)\} \dd u. 
	\end{equation}
	Let $e_{n}(\bbeta_l)$ be the corresponding estimating function given by
	\begin{equation}\label{eq:score}
	e_{n}(\bbeta_l) = \frac{\dd}{\dd \bbeta_l} \ell_{n}(\bbeta_l)=\sum_{u\in {\bX_n}} \bz_l(u) -\theta_n \int_{W_n} \bz_l(u) {\rho_n}(u;\bbeta_l) \dd u.
	\end{equation}
	For any $\bbeta_l \in {\R^{p_l}}$, the sensitivity (or Fisher information) matrix is
	\begin{equation}\label{eq:sensitivity}
	\bS_n(\bbeta_l)  = -\EE  \left\{\frac{\dd}{\dd \bbeta_l^\top} e_n(\bbeta_l)  \right\}= \theta_n \int_{W_n}  \bz_l(u)\bz_l(u)^\top {\rho_n}(u;\bbeta_l) \dd u. 
	\end{equation}
	We assume $\bS_n(\bbeta_l)$ is positive definite for all  $\bbeta_l$ (see also condition~C\ref{c:eigen} in the next section). We can then define the estimator of $\bbeta_l$ as
	\begin{equation}
			\label{eq:bbetalEst}
			\hat \bbeta_{l,n} = \argmax_{\bbeta_l \in {\R^{p_l}} } \,  {p_{n}}(\bX_n;\bbeta_l) = \argmax_{\bbeta_l \in {\R^{p_l}} } \, \ell_{n}(\bbeta_l).
	\end{equation}
If $\bX_n$ is indeed a Poisson process with intensity function ${\rho_n}(\cdot;\bbeta_l)$, then $\hat \bbeta_l$ is the maximum likelihood estimator and the sensitivity \eqref{eq:sensitivity} equals the observed information matrix $-\dd e_n(\bbeta_l)/\dd \bbeta_l^\top$.

If $\bX_n$ is not Poisson, $\hat \bbeta_{l,n}$ may be viewed as a composite likelihood estimator \citep{schoenberg:05,waagepetersen:07}. In the situation where the intensity function ${\rho_n}(\cdot;\bbeta_l)$ coincides with the true intensity function $\lambda_n$, asymptotic properties of maximum likelihood or composite likelihood estimators obtained as maximizers of Poisson likelihood functions have been established in various settings by \cite{rathbun:cressie:94}, \cite{waagepetersen:07}, \cite{guan:loh:07} and \cite{waagepetersen:guan:09}. In the next section we investigate the more intriguing situation where the intensity model is misspecified.

\subsection{Framework and asymptotic results for misspecified intensity functions} \label{sec:asymp}

To handle the situation where ${\rho_n}(\cdot;\bbeta_l)$ does not coincide with the intensity function of $\bX_n$, we follow \cite{varin:vidoni:05} and define a (composite) Kullback-Leibler divergence between the model ${\cal M}_l$ with parameter $\bbeta_l$ and the true sampling distribution. That is, 
\begin{equation}\label{eq:kl} \mathrm{KL}_n(\bbeta_l)= \EE \left\{\ell_n- \ell_n(\bbeta_l) \right\} \end{equation}
where $\ell_n$ is the Poisson log-likelihood obtained with the true intensity $\lambda_n$.   For a window $W_n$ and model ${\cal M}_l$ we let
\[ \bbeta_{l,n}^* = \argmax_{{\bbeta_l \in \R^{p_l}}}  \EE  \ell_n(\bbeta_l) \]
denote the `least wrong parameter value' under model ${\cal M}_l$, provided the maximum exists. It is easy to see by explicit evaluation of the right hand side that
\[
	-\frac{\dd }{\dd {\bbeta_l}^\top} \left\{ -\frac{\dd}{\dd \bbeta_l } \EE \ell_n(\bbeta_l)\right\} = \mathbf S_n(\bbeta_l)	
\]
and so condition~C\ref{c:eigen} stated below implies that $\bbeta_{l,n}^*$ is well-defined as a unique maximum when $n$ is large enough. Also it is easy to see that
\begin{equation}\label{eq:unbiased} \EE e_n(\bbeta_{l,n}^*) =0, \end{equation}
which means that $\hat \bbeta_{l,n}$ given by~\eqref{eq:bbetalEst} is a candidate to estimate $\bbeta_{l,n}^*$.

The remainder of this section is devoted to asymptotic results for $\hat \bbeta_{l,n}$ within the above framework of a misspecified intensity function. We thereby extend the results in the references mentioned in Section~\ref{sec:correct}. In contrast to these references which used either increasing domain or infill asymptotics,  we moreover consider a `double asymptotic' framework as formalized by condition~C\ref{c:asymptotic} presented below.

Two matrices are crucial for the asymptotic results. The sensitivity matrix $\mathbf S_n(\bbeta_l)$ is given by~\eqref{eq:sensitivity} regardless of whether the model is misspecified or not. The variance-covariance matrix $\boldsymbol{\Sigma}_{l,n}$ of \eqref{eq:score} is, using the Campbell theorem, 
given by
\begin{align}
	\boldsymbol{\Sigma}_{l,n} &= \int_{W_n} \bz_l(u) \bz_l(u)^\top\lambda_n(u) \dd u \nonumber \\&\qquad + \int_{W_n}\int_{W_n} \bz_l(u) \bz_l(u)^\top\lambda_n(u) \lambda_n(v) \left\{g_n(u,v) -1\right\}\dd u \dd v. \label{eq:Sigmanl}
\end{align}
Observe that $\boldsymbol \Sigma_{l,n}$ does not depend on $\bbeta_l$ (whence its notation).
Our results will be based on the following assumptions where for a
square matrix $\mathbf M$, $\nu_{\min}(\mathbf M)$ (resp.\
$\nu_{\max}(\mathbf M)$) stands for the smallest (resp.\ largest)
eigenvalue. We use  $ a_n \asymp b_n$ to denote that $a_n={\mathcal
  O}(b_n)$ and $b_n = {\mathcal O}(a_n)$.

\begin{enumerate}
\renewcommand{\theenumi}{\arabic{enumi}}
\renewcommand{\labelenumi}{[C\theenumi]}
\item \label{c:compact} As $n\to \infty$, $\sup_{n \ge 1} \|\bbeta_{l,n}^*\| =\mathcal O(1)$.

\item \label{c:boundedcov} $\bz_l:\R^d \to \R$ is continuous, $0< \inf_{u\in \R^d} \|\bz_l(u)\| <\sup_{u\in \R^d} \|\bz_l(u)\|<\infty$.
  
\item \label{c:asymptotic} The sequence $\left(\tau_n:=\theta_n
    |W_n|\right)_{n\ge 1}$ is an increasing sequence, such that\linebreak
  $\liminf_{n\to \infty} \theta_n >0$ and $\lim_{n\to \infty} \tau_n
  =\infty$. The sets $W_n$ are convex and compact. 

\item \label{c:as} As $n\to \infty$, $\tau_n^{-1} e_n(\bbeta_{l,n}^*) \to 0$ almost surely.
	
\item \label{c:eigen} For any $\bbeta_l \in \R^{p_l}$ and $n\ge 1$, $\bS_n(\bbeta_l)$ is positive definite. In addition, 
  for any $n\ge 1$, there exists a 
  set $B_n \subseteq W_n$  such that $|B_n|\asymp |W_n|$ 
 and a $c>0$ such that $\inf_{n} \inf_{\boldsymbol \phi\in \R^{p_l}, \|\boldsymbol \phi\|=1} \inf_{u\in B_n} |\boldsymbol\phi^\top \bz_l(u) | \ge c$. Finally, we assume that  $\liminf_{n\to \infty} 
\nu_{\min }\left( \tau_n^{-1}  \boldsymbol{\Sigma}_{l,n}
\right)>0$. 

\item \label{c:variance} {As $n \to \infty$,} $\|\boldsymbol \Sigma_{l,n}\| = \mathcal O(\tau_n)$.

\item \label{c:clt} As $n\to \infty$, $\boldsymbol \Sigma_{l,n}^{-1/2}
  e_n(\bbeta_{l,n}^*) \to N(0,\mathbf I_{p_l})$ in distribution.
\end{enumerate}

We can then state the following asymptotic result which is verified in Appendix~\ref{sec:mlestar}.
\begin{theorem} \label{thm:lwpv} ${ }$ \\

\noindent(i) Assume conditions C\ref{c:compact}-C\ref{c:as} hold, then there exists $v=v_n(\hat \bbeta_{l,n},\bbeta_{l,n}^*)$,
such that  almost surely 
\begin{equation}
	\label{eq:betaEst}
	\hat \bbeta_{l,n} -\bbeta_{l,n}^* \; = \; \frac{\bz_l(v)}{\|\bz_l(v)\|^2} \, \log \left[
1+ \tau_n^{-1} {e_n(\bbeta_{l,n}^*)}^\top \frac{\bz_l(v)}{\|\bz_l(v)\|^2} \exp\{ - {\bbeta_{l,n}^*}^\top \bz_l(v)\}
	\right].
\end{equation}
and $\hat \bbeta_{l,n}-\bbeta_{l,n}^* \to 0$ almost surely as $n\to \infty$. \\
(ii) Assume conditions C\ref{c:compact}-C\ref{c:asymptotic} and C\ref{c:eigen}-C\ref{c:variance} hold. Then, $\hat \bbeta_{l,n}$ is a root-$\tau_n$ consistent estimator of $\bbeta_{l,n}^*$, i.e.\
			\begin{equation} \label{eq:roottaunstar}
			\hat \bbeta_{l,n} -\bbeta_{l,n}^* =  \mathcal O_P( \tau_n^{-1/2}).   
			\end{equation}

\noindent (iii) If in addition, C\ref{c:clt} holds, then as $n\to \infty$,
\begin{equation} \label{eq:cltbetastar}
\boldsymbol \Sigma_{l,n}^{-1/2}   \bS_n(\bbeta_{l,n}^*) \left( \hat \bbeta_{l,n} -\bbeta_{l,n}^*\right) \to N(0,\mathbf I_{p_l})
\end{equation}
in distribution.
\end{theorem}  

We stress  that Theorem~\ref{thm:lwpv} (ii)-(iii) do not require the strong consistency of $\tau_n^{-1} e_n(\bbeta_{l,n}^*)$ to 0, i.e. condition C\ref{c:as}. We conclude by some remarks regarding the assumptions C\ref{c:compact}-C\ref{c:clt}.
Condition C\ref{c:compact} ensures that the sequence of `least wrong parameter value` does not diverge with $n$.

Condition~C\ref{c:asymptotic} is different from existing conditions as it embraces both of the standard asymptotic frameworks considered in the literature: 
\begin{itemize}
\item infill asymptotics: $W_n=W$ with $W$ a bounded set of $\R^d$ and $\theta_n \to \infty$ as $n\to \infty$.
\item increasing domain asymptotics: $\theta_n=\theta>0$ and $(W_n)_{n\geq 1}$ is a sequence of bounded domains of $\R^d$ such that $|W_n|\to \infty$.
\end{itemize}
It is also valid if both asymptotics are considered at the same time. The assumed convexity in C\ref{c:asymptotic} enables the use of the mean value theorem in the proofs of our theoretical results.

In Condition C\ref{c:boundedcov}, assuming an upper bound for
$\|\bz_l(u)\|$ is quite standard. The upper bound further implies that assuming a
  lower bound is not really restrictive since for each covariate $z_j$
  we can always find some $k$ so that $|z_j(u)+k| \neq 0$ for any $u
  \in \R^d$. Replacing $z_j$ by $z_j+k$ while changing the intercept
  from $\beta_0$ to $\beta_0-\beta_jk$ leaves the model unchanged.
  The continuity assumption is used to prove the strong consistency of $\hat \bbeta_{l,n}$.

Condition C\ref{c:as} is also used to ensure the strong consistency of $\hat \bbeta_{l,n}$. This condition can be seen as a law of large numbers.
In Example~\ref{ex:PC}, we present a class of models where such an assumption is valid under the generalized asymptotic condition C\ref{c:asymptotic}. We can observe that if there exists $p>0$ such that $\EE\{ \|\tau_n^{-1/2}e_n(\bbeta_{l,n}^*\|^p\}=\mathcal O(1)$ and $\sum_{n} \tau_n^{-p/2}<\infty$, then C\ref{c:as} ensues from an application of Borel-Cantelli's lemma. At least in the increasing domain framework, assuming such a moment assumption is quite standard to derive a central limit theorem.

Condition C\ref{c:eigen} is very similar to the assumption required by
\cite{rathbun:cressie:94} under the Poisson case or by \cite{waagepetersen:guan:09} for more general point processes, within the increasing domain asymptotic
framework. 
Note in particular that C\ref{c:eigen} combined with C\ref{c:compact}-C\ref{c:boundedcov} ensures that $\liminf_{n\to \infty} \nu_{\min}\{\tau_n^{-1} \bS_n(\bbeta_{l,n}^*) \}>0$.

Under the Poisson case, $g_n=1$ and so C\ref{c:variance} is obviously satisfied if \linebreak $\sup_{u\in W_n} \lambda_n(u) = \mathcal O(\theta_n)$. For more general point processes, assume further that $g_n=g$ does not depend on $n$ and is invariant under translations. Then, thanks to C\ref{c:boundedcov}, the assumption
\[
	\int_{\R^d} \{g(o,w)-1\} \dd w <\infty.
      \]
      will imply C\ref{c:variance}.
This assumption is quite standard and satisfied by a large class of models, see e.g. \citet{waagepetersen:guan:09}. 

Continuing within the increasing domain framework, C\ref{c:clt} was established by \citet{rathbun:cressie:94} in the Poisson case, by \citet{guan:loh:07} and \citet{waagepetersen:guan:09} for $\alpha$-mixing point processes, and by \citet{lavancier:poinas:waagepetersen:20} for determinantal point processes. The infill asymptotic framework was used to establish C\ref{c:clt} in case of Poisson cluster processes in \citet{waagepetersen:07}. However, the `double' asymptotic framework has never been considered. Below, we provide an example of a model which satisfies C\ref{c:as}, C\ref{c:variance} and C\ref{c:clt} under the new general asymptotic setting C\ref{c:asymptotic}.

\begin{example} \label{ex:PC}
Let $\bC_n$ be a homogeneous Poisson point process on $\R^d$
  with intensity $\theta_n$. Given $\bC_n$, let
  $\bX_{n,c}$, $c \in \bC_n$, be independent inhomogeneous Poisson
point processses on $W_n$ with intensity $\alpha k(u-c)\rho(u)$ where $\alpha>0$, $k$ is a symmetric
density on $\R^d$, and $\rho$ is a non-negative bounded function. Then, $\bX_n = \cup_{c\in \bC_n} \bX_{n,c}$ is an inhomogeneous Poisson cluster point process (with inhomogeneous offspring). It can be shown that $\lambda_n(u) = \theta_n \alpha \rho(u)$ 
and $g_n(u,v) = 1+ {(k \!*\!k)(v-u)}/{\theta_n}$ where the notation
$*$ denotes convolution. 
Assuming that $\sup_u \rho(u)=\mathcal O(1)$, there exists $K\ge 0$ such that
	\begin{align*}
			\| \boldsymbol \Sigma_{l,n}\| &\le K \left\{ \tau_n +
                                           \theta_n^2 \int_{W_n}
                                           \int_{W_n}
                                           |g_n(u,v)-1| \dd u
                                           \dd v 	\right\}\\
			&\le K  \left\{\tau_n+\theta_n \int_{W_n}
                   \int_{W_n} |(k*k)(v-u)| \dd u \dd v \right\}\\
			& \le K  \left\{\tau_n+ \tau_n \int_{\R^d} |(k*k)(w)| \dd w \right\}= \mathcal O(\tau_n).
	\end{align*}
Thus C\ref{c:variance} is satisfied.	
In Appendix~\ref{app:cltICPP} we show that this model also satisfies C\ref{c:as} and
C\ref{c:clt} within the `double' asymptotic framework. Along the
same lines one can show that C\ref{c:clt} holds for the inhomogeneous
Poisson point process with intensity function $\lambda_n(u)=\theta_n\rho(u)$.
\end{example}

\section{Akaike  and composite information criteria}\label{sec:cic}

One criterion for model selection would be to choose the model that
minimizes the Kullback-Leibler divergence, i.e.\ the model for which
\eqref{eq:kl} evaluated at $\bbeta_{l,n}^*$ is smallest. Of course
this criterion is not useful in practice since the true model is unknown.
Following \cite{varin:vidoni:05} we instead choose the model that
minimizes an estimate of the expected value of \eqref{eq:kl}
evaluated at $\hat \bbeta_{l,n}$, or equivalently, that minimizes an estimate of $2\EE C_n(\hat \bbeta_{l,n})$ with $C_n(\bbeta_l)=- \EE \ell_n(\bbeta_l)$. We follow \citet[Lemmas 1-2]{varin:vidoni:05} to derive in our context the following result.

\begin{proposition}\label{prop:varinvidoni}
Assume conditions C\ref{c:compact}-C\ref{c:asymptotic} and C\ref{c:eigen}-C\ref{c:variance} hold. Also assume that there exists $\varepsilon>0$ such that
\begin{equation} \label{eq:assumptionUI}
	\sup_{n}\EE \left\{ \; \left\| e_n(\bbeta_{l,n}^*)^\top
              \mathbf M_n  e_n(\bbeta_{l,n}^*)
            \right\|^{1+\varepsilon} \; \right\} < \infty
\end{equation}
where $\mathbf M_n = \bS_n(\tilde \bbeta_{l,n})^{-1} - \bS_n(\bbeta_{l,n}^*)^{-1}$ and $\tilde \bbeta_{l,n}$ is on the line segment between $\hat \bbeta_{l,n}$ and $\bbeta_{l,n}^*$.
Then,
\[
\EE \left\{ C_n(\hat \bbeta_{l,n}) \right\} = \EE\left\{ -\ell_n(\hat \bbeta_{l,n})\right\} +
\mathrm{trace} \left\{ \bS_n(\bbeta_{l,n}^*)^{-1} \bSigma_{l,n}\right\} + o(1).
\]
In other words, $-2\ell_n(\hat
\bbeta_l)+2\mathrm{trace}\{\bS_n(\bbeta_{l,n}^*)^{-1} \boldsymbol
  \Sigma_{l,n}\}$ is an asymptotically unbiased estimator of $2\EE \{
C_n(\hat \bbeta_{l,n}) \}$.
\end{proposition}

\begin{remark} The technical condition~\eqref{eq:assumptionUI} implies
    that the sequence \linebreak$\{ \tilde e_n :=e_n(\bbeta_{l,n}^*)^\top \mathbf M_n
    e_n(\bbeta_{l,n}^*)\}_{n\ge 1}$   is a uniformly integrable
    sequence of random variables. This ensures  that convergence in
    probability of $\tilde e_n$ to 0 implies that $\EE \tilde e_n=o(1)$.
\end{remark}

To estimate the effective degrees of freedom
$p^*_l=\text{tr}\left\{\bS_n(\bbeta_{l,n}^*)^{-1} \boldsymbol \Sigma_{l,n} \right\}$ we first
simply estimate $\bS_n(\bbeta_{l,n}^*)$ by $\bS_n(\hat \bbeta_{l,n})$. The
estimation of $\boldsymbol \Sigma_{l,n}$ is more difficult. Following \citet[p.\ 31]{claeskens:hjort:08}, note that if
${\rho_n}(\cdot;\bbeta_{l,n}^*)$ coincides with the true intensity function
of $\bX_n$, then $\boldsymbol \Sigma_{l,n}$ coincides with
$\boldsymbol  \Sigma_n(\bbeta^*_{l,n})$ given by $\bS_n(\bbeta_{l,n}^*)+ T_2$ where
\[
   T_2:=\int_{W_n^2} \bz_l(u) \bz_l^\top (v)
  {\rho_n}(u;\bbeta_{l,n}^*) {\rho_n}(v;\bbeta_{l,n}^*) \left\{ g_n(u,v)-1 \right\}
  \dd u\dd v.
\]
In this case we get
\begin{align*}
\text{trace} \{\bS_n(\bbeta_{l,n}^*)^{-1} \boldsymbol \Sigma_n(\bbeta_{l,n}^*)\}  
 &		 = \text{trace} \{\mathbf{I}_{p_l} + \bS_n(\bbeta_{l,n}^*)^{-1}  T_2\} \\	
 &	 = p_l + \text{trace} \{\bS_n(\bbeta_{l,n}^*)^{-1} T_2\}.
\end{align*}
We propose to use $p_{l,\text{approx}}^*=p_l+ \text{trace} \{ \bS_n(\bbeta_{l,n}^*)^{-1} T_2 \}$ as an approximation of $p_l^*$.  In practice we replace $\bbeta_{l,n}^*$ by $\hat \bbeta_{l,n}$ and $g_n$ by an estimate obtained by fitting a valid parametric model for $g_n$ and thus obtain $\hat p_{l,\text{approx}}^*$. Our composite model selection criterion then becomes
\begin{equation}\label{eq:cic}    \mathrm{CIC}_l= -2l_n(\hat \bbeta_{l,n})   +2 \hat p_{l,\text{approx}}^*. \end{equation}

For a Poisson process, $g_n=1$ in which case $\hat p_{l,\text{approx}}^*=p_l$ and \eqref{eq:cic} reduces to the popular Akaike's information criterion.
For a clustered point process, $g_n>1$ meaning that $\hat p_{l,\text{approx}}^*>p_l$. Thus we penalize more the complexity of the model in the case of a clustered point process. This seems to make sense since random clustering of points (i.e.\ not due to covariates) may erroneously be picked up by covariates that actually had no effect in the data generating mechanism. Hence there is a greater risk of picking a too complex model for the intensity function in case of a clustered point process than for a Poisson process.


\section{Bayesian information criterion}\label{sec:bic}

The motivation of the Bayesian information criterion (BIC) is quite
different from the derivation of the AIC and CIC criteria considered
in the previous section. A main difference is that there is initially no
reference to a Kullback-Leibler distance or asymptotics related to a
`least false parameter value'. Instead, the true model is considered
to be one of the models ${\mathcal M}_l$ and the idea is to
  choose the  model that has maximum posterior probability within the
  specified Bayesian framework. However, the asymptotic concepts again play a role in order to derive asymptotic expansions of the posterior probabilities. Section~\ref{sec:bicPoisson} covers the Poisson process case. Section~\ref{sec:clBIC} proposes a composite likelihood BIC in
  the case where data is not generated from a Poisson process. Note that in this case, it is still assumed that the true
  intensity function corresponds to one of intensity functions for the
  models ${\mathcal M}_l$.

\subsection{BIC in the Poisson process case} \label{sec:bicPoisson}
	
The BIC criterion (see
e.g. \citet{schwarz:78,lebarbier:mary-huard:06}) defines the
	best model $\mathcal M_{\bic}$ as 
	\begin{equation}
	\mathcal M_{\bic} = \argmax_{\mathcal M_l,l=1,\ldots,2^p} \; \P_n \left(  \mathcal  M_l \mid \bX_n \right).
	\end{equation}
	From Bayes formula, 
	\[
	\P_n ( \mathcal M_l \mid \bX_n) = \frac{p_n(\bX_n|{\mathcal M}_l) \P(\mathcal M_l)}{p_n(\bX_n)} .
	\]
 Letting $p(\bbeta_l|{\mathcal M}_l)$ denote the prior density of
 $\bbeta_l$ given ${\mathcal M}_l$,
\[ p_n(\bX_n|{\mathcal M}_l)=
  \int_{\mathbb R^{p_l}}p_n(\bX_n;\beta_l)p(\bbeta_l|{\mathcal M}_l) \dd
  \bbeta_l \]
since $p_n(\bX_n;\beta_l)$ is the conditional density of $\bX_n$ given $\mathcal M_l$ and $\bbeta_l$.
	We assume that the prior distribution over models is non-informative, so that the BIC criterion defines the best model as
	\begin{equation}\label{eq:bic}
	\mathcal M_{\bic} = \argmax_{{\mathcal M_l}, l=1,\dots,2^p} \; p_n \left(   \bX_n \mid \mathcal M_l\right).
      \end{equation}
      
In principle, one could evaluate the $p_n \left(   \bX_n \mid \mathcal M_l\right)$ using numerical quadrature and then determine $\mathcal M_{\bic}$. However, this is computationally costly and also the need to elicit a specific prior $p(\bbeta_l|{\mathcal M}_l)$ for each model may be a nuisance. Our next result therefore proposes an asymptotic expansion of $\log p_n \left(\bX_n \mid \mathcal M_l\right)$. The methodology is standard (basically a Laplace approximation) and well-known in the literature, see \citet{tierney:kadane:86} or e.g. \citet{lebarbier:mary-huard:06} and the references therein. However, due to our spatial framework and the double asymptotic point of view considered in this paper, the standard results do not apply straightforwardly. Due to the use of asymptotic results we again need to rely on the notions of a true model and the `least false parameter value' $\bbeta_{l,n}^*$.
	
We  impose the following conditions on the prior for $\bbeta_l$ and
the mean $\mu_n=\EE \bX_n$.
\begin{enumerate}
\renewcommand{\theenumi}{\arabic{enumi}}
\renewcommand{\labelenumi}{[C\theenumi]}
\addtocounter{enumi}{7}
\item \label{c:prior} The prior density $p(\bbeta_l\mid \mathcal M_l)$ of
	$\bbeta_l$ given ${\mathcal M}_l$ is continuously
        differentiable on $\R^{p_l}$.  
\item \label{c:mun} $\mu_n \asymp \tau_n$.
\end{enumerate}
Note that $\mu_n$ is the marginal mean of $\bX_n$ under the true
intensity model $\lambda_n$ as discussed in
Section~\ref{sec:modelselection}. Condition~C\ref{c:mun} seems reasonable
since it would hold if $\lambda_n$ coincided with any of the specified
parametric intensity models $\rho_n(\cdot;\bbeta_{l})$. 
We also need to slightly strengthen assumption C\ref{c:compact} and replace it by 
\begin{enumerate}
\renewcommand{\theenumi}{\arabic{enumi}$^\prime$}
\renewcommand{\labelenumi}{[C\theenumi]}
\addtocounter{enumi}{0}
\item \label{c:compact2} As $n\to \infty$, there exists $\bbeta_l^*\in \R^{p_l}$ such that $\lim_{n\to \infty} \bbeta_{l,n}^* =\bbeta_l^*$.
\end{enumerate}

The following result is verified in Appendix~\ref{sec:thmlaplace}
using \citet[Theorem~2]{lapinski:19}, which is a rigorous statement of a multivariate Laplace approximation. 
\begin{proposition} \label{thm:laplace}
Under the conditions C\ref{c:compact2},
C\ref{c:boundedcov}-C\ref{c:eigen} and C\ref{c:prior}-C\ref{c:mun}, we have,
almost surely with respect to the distribution of $\{\bX_n\}_{n
    \ge 1}$,  as $n\to \infty$,
\begin{align}
\log p_n(\bX_n \mid \mathcal M_l) =& \ell_n(\hat \bbeta_{l,n}) - \frac{p_l}{2} \log\left( \mu_n \right) \nonumber\\
& + \frac{p_l}2 \log(2\pi) -\frac12 \log \det \{\mu_n^{-1} \bS_n (\hat \bbeta_{l,n}) \}  \nonumber\\
&+ \log p ( \hat \bbeta_{l,n} \mid \mathcal M_l) +  {\mathcal O} (\mu_n^{-1/2}). \label{eq:expansion}
\end{align}
\end{proposition} 

The criterion \eqref{eq:bic} is defined entirely within the specified
Bayesian framework. Hence no reference to a true model and no need for
asymptotic results. However, this is changed when we derive the
expansion \eqref{eq:expansion} for \eqref{eq:bic}. The expansion is
around $\hat \bbeta_{l,n}$ and for technical reasons, when applying
the Laplace approximation, convergence of $\hat \bbeta_{l,n}$ is
needed. Therefore we need the assumptions that ensure strong
consistency of $\hat \bbeta_{l,n}$.

	Since $\mu_n \asymp \tau_n$ and from the strong consistency of $\hat \bbeta_{l,n}$, we have that \linebreak$\log \det \{ \mu_n^{-1}  \bS_n (\hat \bbeta_{l,n})\}= \mathcal O_\P(1)$ while  C\ref{c:prior} ensures that $\log p (\hat \bbeta_{l,n} \mid \mathcal M_l)$ is $\mathcal O_\P(1)$.
	So, if we neglect terms which are $ \mathcal O_\P(1)$  in~\eqref{eq:expansion}, we follow the standard heuristic and suggest to define a first version of the  BIC criterion as 
	\begin{equation}
	- 2 \ell_n(\hat \bbeta_{l,n}) + p_l \log (\mu_n)
	\end{equation}
	where we remind that $p_l$ is the length of $\bbeta_l$.
	
	In practice $\mu_n$ is not known. However, since $\Var N(W_n)=
	\mu_n$ it follows that
	$N(W_n)/\mu_n -1 = \mathcal O_\P(\mu_n^{-1/2})=o_\P(1)$. This justifies to define the $\bic$ criterion in the following natural way
	\begin{equation} \label{eq:ourbic}
	\bic_l = - 2 \ell_n(\hat \bbeta_{l,n}) + p_l \, \log \left\{ N(W_n) \right\}.
	\end{equation}

\subsection{Composite likelihood BIC}\label{sec:clBIC}

	Suppose $\bX_n$ has an intensity function of the form \eqref{eq:loglinear} but is not a Poisson process. Then as mentioned in Section~\ref{sec:correct}, \eqref{eq:loglikelihood} may be viewed as a composite likelihood score for estimating $\bbeta_l$. In this case, following \cite{gao:song:10}, \eqref{eq:bic} may be viewed as a composite likelihood BIC. Again we obtain \eqref{eq:ourbic} from \eqref{eq:bic} by Laplace approximation. However, \cite{gao:song:10} suggest to replace $p_l$ by the `effective degrees of freedom' $p_l^*$ considered in Section~\ref{sec:cic}. Thus our  proposed composite likelihood BIC is
\[ \cbic_l = - 2 \ell_n(\hat \bbeta_{l,n}) + p_{l,\text{approx}}^* \, \log \big\{ N(W_n)\big\} \]
which becomes equal to the ordinary BIC in \eqref{eq:ourbic} for a Poisson process. From a practical point of view, we simply estimate $p_{l,\text{approx}}^*$ as in~\eqref{eq:cic}.

\section{Simulation study} \label{sec:sim}

To evaluate the proposed model selection criteria we conduct two simulation
studies with $p=6$ spatial covariates, of which four have zero
effect. The first study in Section~\ref{sec:sim:pois} considers AIC
and BIC in the Poisson point process case. 
  The second study considers the clustered Thomas point process where we
employ the CIC and CBIC criteria to select the best model and compare with
results obtained using AIC and BIC (assuming wrongly that the
simulated data are from a Poisson point process).
	
The covariates are obtained from the BCI dataset
\citep{hubbell:foster:83,condit:hubbell:foster:96,condit:98} which in addition to locations
  of around 300 species of trees observed in $W=[0,1000]\times[0,500]$ ($m^2$)  
  contains a number of spatial covariates. In particular, we center and scale the two topological covariates (elevation and slope of elevation) and four soil nutrients (aluminium, boron, calcium, and copper). The six covariates are depicted in Figure~\ref{cov}. 
	
	\begin{figure}[!ht]
		\renewcommand{\arraystretch}{0}
		\setlength{\tabcolsep}{1pt}
		\begin{tabular}{l l l}
			\includegraphics[width=0.33\textwidth]{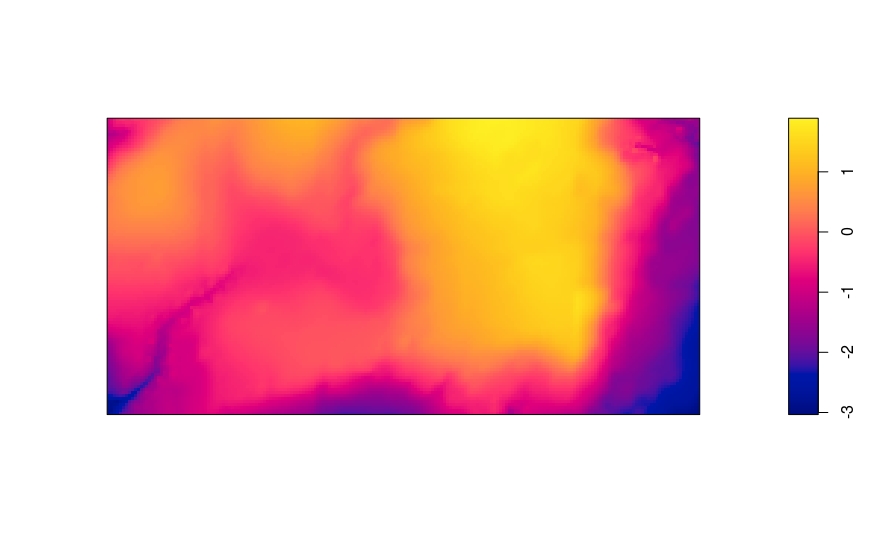} & 	\includegraphics[width=0.33\textwidth]{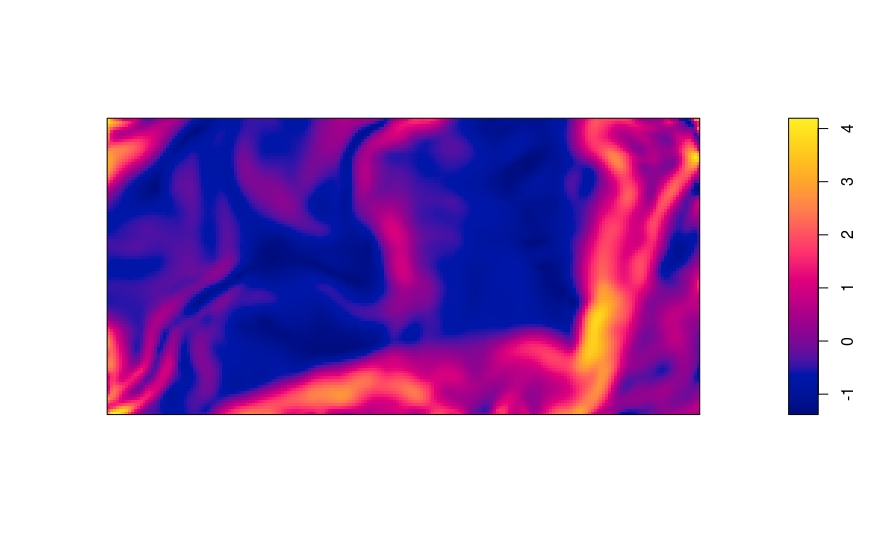} &	\includegraphics[width=0.33\textwidth]{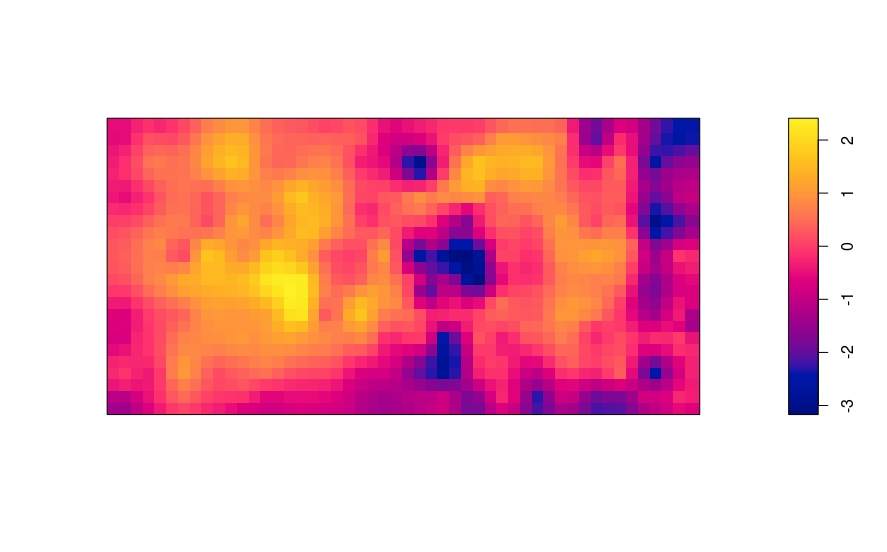}  \\
			\includegraphics[width=0.33\textwidth]{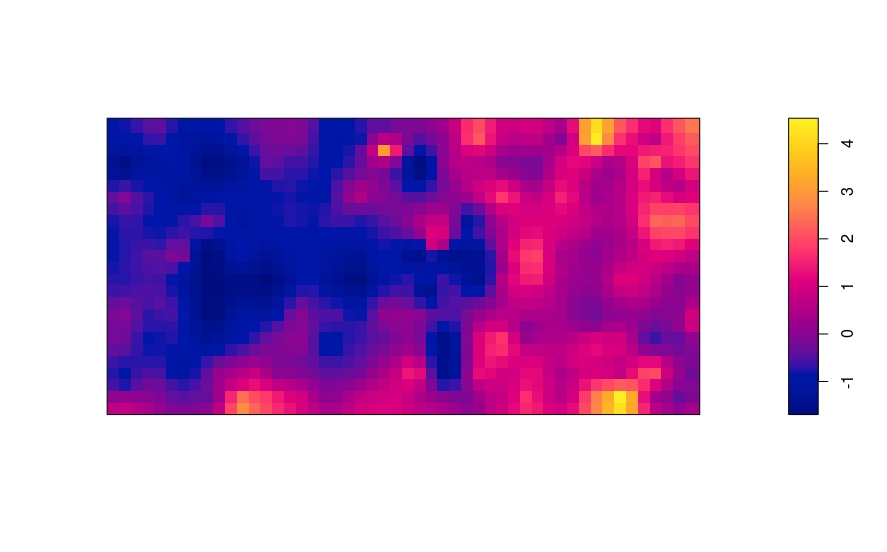} & 	\includegraphics[width=0.33\textwidth]{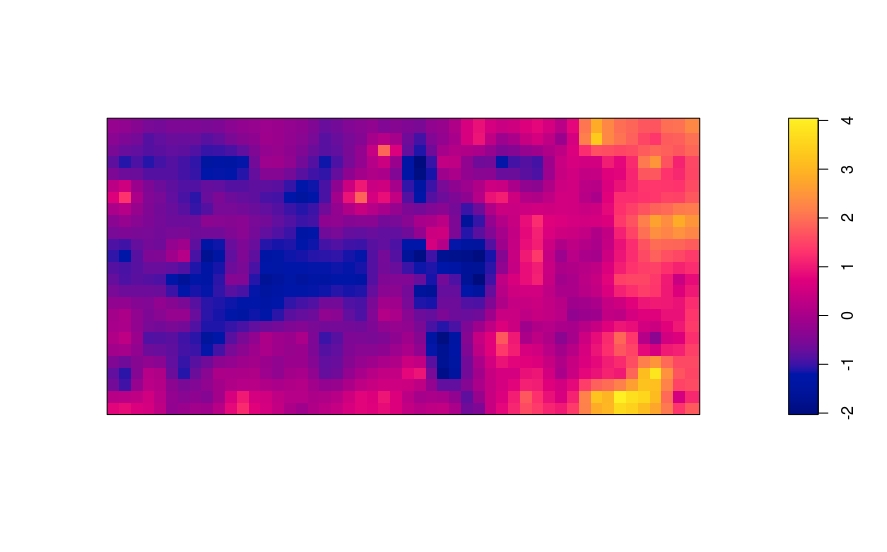} &	\includegraphics[width=0.33\textwidth]{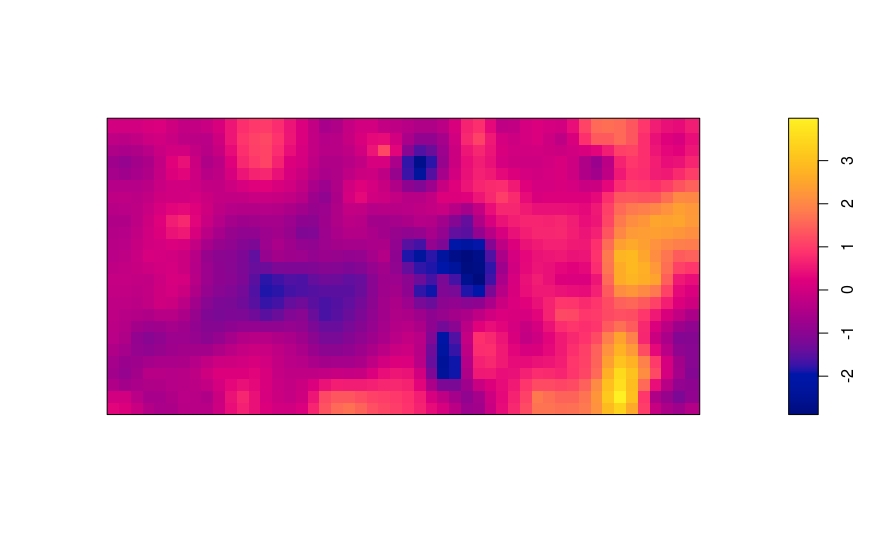}  \\
		\end{tabular}
		\caption{Maps of covariates used in the simulation study. From left to right: First row:
			elevation ($z_1$), slope ($z_2$) and aluminium ($z_3$); Second row: boron ($z_4$), calcium ($z_5$) and copper ($z_6$).}
		\label{cov}
	\end{figure}
	
	Skipping the dependence on $n$ in the notation,  we model the intensity function as
	\begin{align} \label{eq:int:sim}
	\rho(u,\bbeta)=\omega\exp\{\beta_1 z_1(u)+\cdots+\beta_6 z_6(u)\},
	\end{align}
	where $z_1,\cdots,z_6$ are the centered and scaled covariates as in Figure~\ref{cov} and $\beta_1,\cdots,\beta_6$ are the regression coefficients. We consider different settings for $\omega$ and $W$. Note that ${\omega}$ plays the role of $\theta \exp(\beta_0)$ and will be adjusted to obtain desired expected numbers of points $\mu$ in $W$. When the simulation involves an observation window $W$ different from $[0,1000]\times[0,500]$, the covariates are simply rescaled to fit $W$.
	
The model selection criteria are compared in terms of the true
positive rate (TPR),  false positive rate (FPR), expected
Kullback-Leibler divergence (MKL), and mean integrated squared
error of the intensity function (MISE). The TPR (resp.\ FPR) are the
expected fractions of informative (resp.\ non-informative)
covariates included in the selected model. 
For a point process observed on $W$, with intensity $\rho(;\bbeta^*)$
where $\bbeta^*$ stands for the true parameter estimated by $\hat
\bbeta$, MKL and MISE are estimated by  averaging the following KL and ISE across simulations:
	\begin{align*}
	\mathrm{KL}= & \int_{W} \left[\rho(u;\bbeta^*)\left\{\log\rho(u;\bbeta^*)-1\right\}-\rho(u;\hat \bbeta)\left\{\log\rho(u;\hat \bbeta)-1\right\} \right] \mathrm{d}u \\
	\mathrm{ISE}= &\int_{W} \left\{\rho(u;\bbeta^*)-\rho(u;\hat\bbeta)\right\}^2 \mathrm{d}u.
	\end{align*}

	\subsection{Poisson point process model}	\label{sec:sim:pois}
	
	We consider two different scenarios to illustrate both types of asymptotics.
	\begin{itemize}
		\item \noindent{\bf Scenario 1 (infill asymptotics).}  $W=[0,1]\times[0,0.5]$. We adjust $\omega=\theta \exp(\beta_0)$ such that $\mu$ equals either $50$ or $200$. 
		\item \noindent{\bf Scenario 2 (increasing domain
                    asymptotics).}  $W=[0,500]\times[0,250]$ or
                  $[0,1000]\times[0,500]$ and $\omega$ is chosen so that  $\mu=200$ or $\mu=800$. 
	\end{itemize}

For each scenario and the choices of $W$ and $\omega$, 500
  simulations are generated from inhomogeneous Poisson point processes
  with intensity function given by~\eqref{eq:int:sim} using the
function $\mathtt{rpoispp}$ from the $\mathtt{spatstat}$ $\mathtt{R}$
package. We set $\beta_1=0.5$ and $\beta_2=-0.25$ to represent
moderate effects of elevation and slope, and set
$\beta_3=\dots=\beta_6=0$. For each simulation, parameters are
estimated by maximizing the Berman-Turner approximation \citep[see
e.g.\ ][]{baddeley:turner:00}  of the Poisson
log-likelihood~\eqref{eq:loglikelihood} using a number $m$ of quadrature dummy
points to approximate the integral in~\eqref{eq:loglikelihood}. The
estimation is done using the $\mathtt{ppm}$ function. Then, the model is selected according to the AIC and BIC-type criteria. We consider different variants of the BIC criterion, namely 
\[
\mathrm{BIC}_l(\pi) = -2 \ell(\boldsymbol{\hat \beta}_{l})+ p_{l,\text{approx}}^* \; \log(\pi) \quad l=1,\cdots,64,
\]
where $\pi$ represents a penalty. Note that (omitting dependence on $l$) \linebreak
$\mathrm{AIC}=\mathrm{BIC}\{\exp(2)\}$ and that $\mathrm{BIC}(N)$ ($N=N(W)$)
corresponds to the criterion used by~\citet{thurman:fu:guan:15} and is
also the criterion suggested by the present paper. We also consider $\mathrm{BIC}(|W|)$ used by~\citet{choiruddin:coeurjolly:letue:18} and $\mathrm{BIC}(N+m)$ considered by~\citet{jeffrey:etal:18}.

\begin{table}[!ht]
\centering
\setlength{\tabcolsep}{5pt}
\begin{tabular}{llcccc}
\hline
&& AIC & \multicolumn{3}{c}{BIC$(\pi), \;\; \pi=$}\\
&& & {\small$N$} & {\small$N+400\mu$} & {\small$|W|$} \\
\hline
{\small$W=[0,1]\!\!\times\!\! [0,0.5]$} & TPR &70 & 59  & 49 & 100 \\
$\mu=50$ & FPR &16 & 5 &  1 & 100 \\
& MISE& 6.5 & 6.0 & 6.0 & 7.9 \\
& MKL  &3.0 & 2.9  & 2.9 & 3.6\\
\hline
{\small$W=[0,1]\!\!\times\!\! [0,0.5]$} & TPR  & 94 & 83 & 63 & 100 \\
$\mu=200$ & FPR& 17 & 2 & 0 & 100 \\
& MISE  & 2.6 & 2.3 & 3.3 & 3.2 \\
& MKL  & 2.9 & 2.8 & 4.2 & 3.5\\
\hline
{\small$W=[0,500]\!\!\times\!\! [0,250]$} & TPR &95 & 83 & 63 & 62 \\
$\mu=200$ & FPR  & 16 & 2 & 0 & 0 \\
& MISE & 1.0 & 0.9 & 1.3 & 1.3 \\
& MKL & 2.8 & 2.7 & 4.1 & 4.2\\
\hline
{\small$W=[0,1000]\!\!\times\!\! [0,500]$} & TPR & 100 & 99 & 99 & 98\\
$\mu=800$ & FPR & 18 & 1 & 0 & 0 \\
& MISE  & 10.4 & 6.5 & 6.7 & 7.0 \\
& MKL  & 2.9 & 1.8 & 1.9 & 2.0 \\
\hline
\end{tabular}	
\caption{True positive (TPR) and false positive (FPR) rates in   percent, MISE, and MKL, estimated from 500 simulations of
  inhomogeneous Poisson point processes on different observation
  domains. Model selections are based on AIC or BIC criteria of the
  form $\mathrm{BIC}_l = -2 \ell(\hat \bbeta_{l}) + p_l \log \pi$
  ($l=1,\dots,64$) for different penalty terms
  $\pi=N,N+400\mu,|W|$. For convenience, the four rows with MISE are
  multiplied by respectively .001, .01, 100, and 1000.}
\label{tab:poisson}
\end{table}

For both scenarios, we perform estimation 
with $m=4 \mu$ (the rule of thumb suggested by \texttt{spatstat}) and model selection with the criteria $\mathrm{AIC}$, $\mathrm{BIC}(N)$, $\mathrm{BIC}(N+4\mu)$ and $\mathrm{BIC}(|W|)$. Similar results are obtained with $\mathrm{BIC}(N)$ and $\mathrm{BIC}(N+4\mu)$ and we omit the results for the latter. We also perform estimation and selection with $m=400\mu$ and only report results for $\mathrm{BIC}(N+400\mu)$ since the results with the criteria $\mathrm{AIC}$, $\mathrm{BIC}(N)$ and $\mathrm{BIC}(|W|)$  are similar to those obtained when the estimation is performed with $m=4\mu$.

When $|W|$ is small, especially when $\log|W|<0$, the criterion $\mathrm{BIC}(|W|)$ obviously fails as it selects the most complex model regardless of the value of $\mu$. Thereby the FP rate
becomes 100\%. In addition, as indicated in the second and third rows
of Table~\ref{tab:poisson} where the point patterns have the same
average number of points, it is worth noticing that
$\mathrm{BIC}(|W|)$ selects pretty different models. In particular this
	criterion has an undesirable strong dependence on the choice of length unit.

The criterion $\mathrm{BIC}(N+m)$ with a  large $m$ also fails
since the TPR with this criterion and $m=400 \mu$ is very small compared to the
	other criteria, especially when the expected number of points is small or moderate.
The AIC criterion  achieves a high TPR in all situations but fails since it suffers from a high FPR, even in the scenario 2 where
$\mu=800$.  

In all cases, $\mathrm{BIC}(N)$ provides the best trade-off between TPR and
	FPR and the results improve when $\omega$ or $|W|$ is increased. The
	minimal values of MKL and MISE are further always obtained with
	$\mathrm{BIC}(N)$. In case of a Poisson process, we therefore
	recommend $\mathrm{BIC}(N)$ (simply denoted BIC in the following).
	
\subsection{Thomas point process model} \label{sec:sim:nonpois}

To generate a simulation from a Thomas point process with
intensity \eqref{eq:int:sim}, we first generate a
parent point
pattern from a stationary Poisson point process $\mathbf{C}$ with
intensity $\kappa>0$. Given $\mathbf{C}$, clusters
  $\mathbf{X}_c$, $c \in \mathbf{C}$, are generated from inhomogeneous Poisson point processes  with intensity functions
\begin{align*}
\rho_{c}(u; \bbeta)=\omega\exp\{\beta_1 z_1(u)+\cdots+\beta_6 z_6(u)\} k(u-c;\gamma)/\kappa,
\end{align*}	
where $k(u-c;\gamma)=(2 \pi \gamma^2)^{-1} \exp(-\|u-c\|^2/(2
\gamma^2))$ is the density for $\mathcal{N}(0,\gamma^2
\mathbf{I}_2)$. Finally, $\mathbf{X}=\cup_{c \in
  \mathbf{C}}\mathbf{X}_c$ is an inhomogeneous Thomas point process
with intensity \eqref{eq:int:sim}. The regression parameters are set
as follows: $\beta_1=2$, $\beta_2=-1$, $\beta_3=\dots=\beta_6=0$. We
consider $\kappa=4 \times 10^{-4}$ and two scale parameters $\gamma=5$
and $\gamma=15$.  A lower value for $\gamma$ tends to produce
more clustered patterns. We consider the observation domains
$W=[0,500]\times[0,250]$ and $W=[0,1000]\times[0,500]$ with $\omega$
adjusted to give expected numbers of points $\mu=400$ and $\mu=1600$
for the two windows. The chosen value of $\kappa$ implies on average
50 parent points on $W=[0,500]\times[0,250]$ and 200 parent points on $W=[0,1000]\times[0,500]$.
\begin{table}[!ht]
	\centering
	\begin{tabular}{llrrrr}
		\hline
&		& AIC & BIC & CIC & CBIC \\ 
		\hline
{$W=[0,500]\!\!\times\!\! [0,250]$} &		TPR &98 & 96 & 88 & 81\\
$\mu=400$	&	FPR& 81 & 65 & 24 & 17 \\
$\gamma=5$	&	 MISE & 4.5 & 4.4 & 2.7 & 2.5 \\
	&	MKL & 9.7 & 9.6 & 7.5 & 9.1 \\
	&	Mean$(\hat p_l^*)$ & - & -  & 103.6 & 76.6 \\
	&	SD$(\hat p_l^*)$ & - & - & 122.7 & 70.1 \\
		\hline
{$W=[0,1000]\!\!\times\!\! [0,500]$} &		TPR & 100 & 100 & 97 & 94 \\
$\mu=1600$	&	FPR & 81 & 65 & 20 & 11 \\
$\gamma=5$	&	MISE & 3.9 & 3.8 & 2.6 & 2.4 \\
	&	MKL & 10.3 & 10.2 & 7.9 & 8.2 \\
&	Mean$(\hat p_l^*)$ & - & - & 101.1 & 81.3 \\
	&	SD$(\hat p_l^*)$ & - & - & 104.6 & 70.2 \\
		\hline
\hline
{$W=[0,500]\!\!\times\!\! [0,250]$} &		TPR & 100 & 99 & 95 & 93 \\
$\mu=400$	&	FPR & 70 & 49 & 43 & 32 \\
$\gamma=15$	&	MISE & 2.1 & 2.0 & 1.8 & 1.8 \\
	&	MKL & 5.2 & 5.0 & 4.9 & 5.4 \\
&	Mean$(\hat p_l^*)$ & - & - & 48.9 & 40.9 \\
	&	SD$(\hat p_l^*)$ & - & - & 77.3 & 61.3 \\
		\hline
{$W=[0,1000]\!\!\times\!\! [0,500]$} &		TPR & 100 & 100 & 96 & 94 \\
$\mu=1600$	&	FPR & 78 & 57 & 36 & 24 \\
$\gamma=15$	&	MISE & 2.8 & 2.8 & 2.3 & 2.5 \\
	&	MKL & 7.7 & 7.5 & 7.3 & 9.3 \\
&	Mean$(\hat p_l^*)$ & - & - & 96.1 & 78.1 \\
	&	SD$(\hat p_l^*)$ & - & - & 172.1 & 124.6 \\
		\hline
	\end{tabular}
	\caption{True positive (TPR) and false positive (FPR) rates in percent, MISE, MKL (divided by 10), and mean and standard deviations of estimates of $p_{l,\text{approx}}^*$, based on 500 simulations from inhomogeneous Thomas point processes with $\kappa=4\times10^{-4}$ and scale parameter $\gamma=5$ or 15, observed on different observation domains. Parameters are adjusted to have on average $\mu=400$ points in the small window and 1600 points in the larger one. 
Model selection is  based on AIC, BIC, CIC, and CBIC.}
\label{tab:thomas}
\end{table}

The \texttt{R} function \texttt{rThomas} is used to simulate the point patterns and the function $\mathtt{kppm}$ to estimate the regression parameters. We use $m=16\mu$ dummy points for the different integral approximations.
Models are then selected using four criteria: AIC, BIC, CIC and
CBIC. When using AIC or BIC we implicitly assume wrongly that
the simulated patterns come from a Poisson point process and
thus we set $p_{l,\text{approx}}^*=p_l$. For the composite likelihood
type criteria CIC and CBIC, we compute $\hat p_{l,\text{approx}}^*$
using the \texttt{R} function \texttt{vcov.kppm}. The parameters
$\kappa$ and $\gamma$ are estimated using minimum contrast estimation
with the tuning parameter $r_{\max}$ \citep{waagepetersen:guan:09} set to 20 when $\gamma=5$ and to 50 when $\gamma=15$.

Results are reported in Table~\ref{tab:thomas}. The AIC and BIC
criteria ignore the second-order structure of the
simulated point patterns and we observe that overall AIC and
  BIC produce  high TPR but also very high FPR. The CIC and CBIC give
much more reasonable trade-offs between TPR and FPR and also (with one exception) give smaller MKL and MISE than AIC and BIC.

Focusing on CIC and BIC, we first notice that the means of the estimates of
$p_{l,\text{approx}}^*$ are high. This is because the
considered clustered point processes are far from  Poisson models. We
also notice that even when the number of points is quite large, the
estimation of $p_{l,\text{approx}}^*$ is inaccurate with standard
deviations of the estimates of the same order as the
means. Nevertheless, the resulting CIC and CBIC give reasonable
results. Comparing CIC and CBIC, CIC in general has a higher FPR than CBIC
but on the other hand always gives the smallest MKL. The TPR and MISE are quite
similar for CIC and BIC. Hence in the case of the clustered point process
considered here, CIC and CBIC clearly outperform AIC and BIC but there is not a clear winner between CIC
and CBIC.
	
\section{Discussion} \label{sec:concl}

In this paper we establish a theoretical foundation for various
model selection criteria for inhomogeneous point processes under various asymptotic settings. In case of
a Poisson process a main contribution is to identify in relation to
BIC, the correct
interpretation of `sample size' which based on our
theoretical derivation is the expected number of points which in
practice is estimated by the observed number of points. This
interpretation is supported by our simulation study which also
supports the common understanding that BIC may be preferable to AIC
which tends to pick too complex models.

 More generally for selecting a regression model for the intensity function of a
general point process we develop composite model selection criteria,
CIC and CBIC that clearly outperform AIC and BIC in the simulation
study for a clustered point process. One issue regarding CIC and BIC
is to estimate the bias correction for the estimate of the composite
Kullback-Leibler divergence which depends on the unknown true intensity
function and pair correlation function. Here, inspired by the approach
underlying AIC, for a given model we simply plug in the fitted
intensity function and pair correlation function for the model in
question. This is computationally convenient and we leave it as an
open problem to develop more precise estimates.

Further interesting topics for future research would be to study the
theoretical foundation of criteria for selecting models for the
conditional intensity of a Gibbs process fitted by
pseudo-likelihood. Another interesting problem is selection of the
penalization parameter when the intensity function is estimated using
regularization methods like the lasso \cite[see e.g][]{thurman:fu:guan:15,choiruddin:coeurjolly:letue:18}. This is not covered by our theoretical results which rely on asymptotic results for unbiased estimating functions or Bayesian considerations.\\[\baselineskip]
{\bf Acknowledgments}\\[\baselineskip]
The research of J.-F. Coeurjolly is supported by the Natural Sciences and Engineering Research Council.Rasmus Waagepetersen is supported by The Danish Council for Independent Research | Natural Sciences, grant DFF - 7014-00074 "Statistics for point processes in space and beyond", and by the Centre for Stochastic Geometry and
Advanced Bioimaging, funded by grant 8721 from the Villum Foundation.

	\appendix 
	
	\section{Proof of Theorem~\ref{thm:lwpv}} \label{sec:mlestar}

\begin{proof}
(i) Using a zero order Taylor expansion of $e_n(\hat \bbeta_{l,n})$ (which equals zero by definition) around $\bbeta_{l,n}^*$ expressed with integral remainder term, we have
\begin{align*}
 e_n(\bbeta_{l,n}^*) = & \left(\int_{0}^1   \!\!\int_{W_n}\!\!\! \theta_n \bz_l(u) \bz_l(u)^\top \exp \left[ 
	\left\{ \bbeta_{l,n}^* + t(\hat \bbeta_{l,n}-\bbeta_{l,n}^*) \right\}^\top \bz_l(u)\right] \dd u \dd t \right) \\
	&\times \left( \hat \bbeta_{l,n} - \bbeta_{l,n}^* \right).
\end{align*}
By rearranging the right-hand side of the latter equation and using Fubini's theorem, we have 
\begin{align*}
e_n(\bbeta_{l,n}^*) =& \theta_n \bigg[\int_{W_n} \bz_l(u) \exp \left\{ {\bbeta_{l,n}^*}^\top \bz_l(u)  \right\} \times \\
&\int_{0}^1	\left\{(\hat \bbeta_{l,n} - \bbeta_{l,n}^*)^\top \bz_l(u) \right\} \exp \left\{ t (\hat \bbeta_{l,n} - \bbeta_{l,n}^*)^\top \bz_l(u)\right\} \dd t \bigg] \dd u \\
=& \theta_n \int_{W_n} \bz_l(u) \exp \left\{ {\bbeta_{l,n}^*}^\top \bz_l(u)  \right\} 
\left[
\exp \left\{ (\hat \bbeta_{l,n} - \bbeta_{l,n}^*)^\top \bz_l(u)\right\} -1
\right] \dd u.
\end{align*}
Now, using 
C\ref{c:boundedcov}-C\ref{c:asymptotic}, we can apply a mean value theorem for multiple integrals: there exists $v=v_n(\hat \bbeta_{l,n},\bbeta_{l,n}^*) \in W_n$ such that
\begin{equation}
	\label{eq:tmp}
	e_n(\bbeta_{l,n}^*) = \tau_n \bz_l(v) \exp \left\{ {\bbeta_{l,n}^*}^\top \bz_l(v)  \right\} 
\left[
\exp \left\{ (\hat \bbeta_{l,n} - \bbeta_{l,n}^*)^\top \bz_l(v)\right\} -1.
\right]
\end{equation}
Since  $\|\bz_l(v)\|>0$ almost surely by condition C\ref{c:boundedcov}, we  deduce~\eqref{eq:betaEst} from~\eqref{eq:tmp} using a little algebra. Let
\[ A
_n = \tau_n^{-1} {e_n(\bbeta_{l,n}^*)}^\top \frac{\bz_l(v)}{\|\bz_l(v)\|^2} \exp\{ - {\bbeta_{l,n}^*}^\top \bz_l(v)\}.\] 
By conditions C\ref{c:compact}-C\ref{c:as}, it is clear that $A_n \to 0$ almost surely as $n \to \infty$. Using the continuity of $t \mapsto \log(1+t)$ and again condition C\ref{c:boundedcov}, we deduce that $\log(1+A_n)\bz_l(v)/\|\bz_l(v)\|^2$ tends to 0 almost surely.\\
 
(ii) The proof consists in applying a modified version of 
\citet[Theorem~2]{waagepetersen:guan:09} to the sequence of estimating
functions $e_n(\bbeta_l)$. The modification is given in Appendix~\ref{sec:wg09} and is needed to handle a sequence of `least false' parameter values $\boldsymbol\ta_n^*$ instead of a unique `true' value $\boldsymbol\ta^*$. Thus, we need to prove the assumptions G\ref{cond:tozero}-G\ref{cond:unormal} in Appendix~\ref{sec:wg09}. 
We leave the reader to check that conditions C\ref{c:asymptotic} and~C\ref{c:eigen}  imply conditions G\ref{cond:tozero} and G\ref{cond:posdef}  with $\mathbf V_n=\sqrt{\tau_n}\mathbf I_p$ and $\mathbf J_n( \boldsymbol\theta^*_n)= \mathbf S_n(\bbeta_{l,n}^*)$.
			
To prove assumption G\ref{cond:contgeneral}, we have to prove that as $n\to \infty$,
\[
\sup_{\sqrt{\tau_n }\|\bbeta_l-\bbeta_{l,n}^*\|\leq d} \frac{\|\bS_n(\bbeta_l)-\bS_n(\bbeta_{l,n}^*)\|_M}{\tau_n} \to 0,
\]
where $\|\mathbf M\|_M= \max_{ij} |M_{ij}|$. 
Consider a $\bbeta_l$ such that $\sqrt{\tau_n} \|\bbeta_l -\bbeta_{l,n}^*\|\le d$. Under conditions~C\ref{c:compact}-C\ref{c:asymptotic}, there exist $K_1,K_2,K<\infty$ such that
\begin{align}
\| \bS_n(\bbeta_l) -\bS_n(\bbeta_{l,n}^*)\|_M &= \| \theta_n
                                                       \int_{W_n}
                                                       \bz_l(u)
                                                       \bz_l(u)^\top
                                                       \left\{
                                                       \rho(u;\bbeta_l)-\rho(u;\bbeta_{l,n}^*)\right\}
                                                       \dd u	\|_M \nonumber\\
&\le K_1 \, \tau_n \|\bbeta_l-\bbeta_{l,n}^*\| \exp \left\{  K_2 \left(\|\bbeta_{l,n}^*\| +d/\sqrt{\tau_n} \right) \right\} \nonumber \\
&\leq K \tau_n \|\bbeta_l-\bbeta_{l,n}^*\| =\mathcal O(\sqrt{\tau_n}) = \mathcal O_\P(\sqrt{\tau_n}). \label{eq:diffSnstar}
\end{align}
To verify condition G\ref{cond:unormal}, it is sufficient to note that condition~C\ref{c:variance} implies that the variance of $\tau_n^{-1/2} e_n(\bbeta_{l,n}^*)$ is bounded. Letting $\hat
\bbeta_{l,n}$ denote the (unique for $n$ large enough) solution of
$e_n(\bbeta_l)=0$ or, equivalently,  
\[
\hat \bbeta_{l,n}  = \argmax_{\bbeta_l \in \R^{p_l}} \ell_n(\bbeta_l)
\]
we can then conclude that $\hat \bbeta_{l,n}$ is  a root-${\tau_n}$
consistent estimator of $\bbeta_{l,n}^*$, that is, $\sqrt{\tau_n}(\hat \bbeta_{l,n}-\bbeta_{l,n}^*)$ is
bounded in probability, which proves~\eqref{eq:roottaunstar}. 
		
(iii)  We use a Taylor expansion around $\bbeta_{l,n}^*$: there exists $t\in (0,1)$ and $\tilde \bbeta_{l,n} = \hat \bbeta_{l,n} + t(\bbeta_{l,n}^*- \hat \bbeta_{l,n})$ such that
\begin{align*}
e_n(\bbeta_{l,n}^*)&= e_n(\bbeta_{l,n}^*) - e_n(\hat \bbeta_{l,n}) = \bS_n(\tilde \bbeta_{l,n}) (\hat \bbeta_{l,n} - \bbeta_{l,n}^*)	\\
&= \bS_n(\bbeta_{l,n}^*) (\hat \bbeta_{l,n} - \bbeta_{l,n}^*) + \mathbf B_n \bS_n^\prime
\end{align*}
where $\mathbf B_n = \sqrt{\tau_n} (\hat \bbeta_{l,n}-\bbeta_{l,n}^*)$
and where $\bS_n^\prime = \tau_n^{-1/2} \left\{  \bS_n(\tilde
  \bbeta_{l,n}) - \bS_n(\bbeta_{l,n}^*) \right\}$. We now show
  that $\mathbf B_n \mathbf S_n^\prime= \mathcal O_\P(1)$. Let $d$ and $K$ be given
  as above and let $K_3 \ge d^2K$. Using~\eqref{eq:diffSnstar} we obtain
\begin{align*}
& \P \left( \|\mathbf B_n \mathbf S_n^\prime\|_M \ge K_3 \right ) \leq
                 \P    \left(\| \mathbf B_n \mathbf S_n^\prime\|_M \ge
                 K_3 , \|\mathbf B_n\|<  d\right) + \P \left(
                 \|\mathbf  B_n\|\ge d \right) \\
\leq  &	\P    \left( d K \|\mathbf B_n\| \ge K_3 \right) + \P \left(
        \|\mathbf B_n\|\ge d \right) \le 2 P \left (\|\mathbf B_n\| \ge d \right)
\end{align*}
Now from (ii), for any $\varepsilon>0$, by choosing $d$ large
enough, $P(\|\mathbf B_n\| \ge d) \le \varepsilon/2$ for $n$ sufficiently
large whereby $\P \left( \|\mathbf B_n \mathbf S_n^\prime\|_M \ge K_3 \right)
  \le \varepsilon$ for $n$ sufficiently large.

Finally,
since by condition C\ref{c:eigen}, $\|\bSigma_{l,n}^{-1/2}\| =  \mathcal O(\tau_n^{-1/2})$, we conclude that
\begin{align*}
\boldsymbol\Sigma_{l,n}^{-1/2} \bS_n(\bbeta_{l,n}^*) \left( \hat \bbeta_{l,n} -\bbeta_{l,n}^* \right)  &= \boldsymbol\Sigma_{l,n}^{-1/2} e_n(\bbeta_{l,n}^*) + \boldsymbol\Sigma_{l,n}^{-1/2} \mathcal O_P(1)\\
& = \boldsymbol\Sigma_{l,n}^{-1/2} e_n(\bbeta_{l,n}^*) + o_\P(1)	
\end{align*}
which yields the result using   condition~C\ref{c:clt} and Slutsky's lemma.
\end{proof}

\section{Conditions~C\ref{c:as} and~C\ref{c:clt} for the inhomogeneous Poisson cluster point process}
\label{app:cltICPP}

In this section, we show that the inhomogeneous Poisson cluster point
process presented in the end of Section~\ref{sec:asymp} satisfies
C\ref{c:as} and~C\ref{c:clt}. Recall $\lambda_n(u)=
  \theta_n\alpha \rho(u)$ where $\sup_u \rho(u)={\mathcal O}(1)$.

\begin{proof}
For $k=1,\dots,\lfloor \tau_n \rfloor$, let $\bC_{n,k}$ be independent inhomogeneous Poisson point processes with intensity $\theta_n/\lfloor \tau_n \rfloor$. By the property of any Poisson point process, $\bC_n$ has the same distribution as $\cup_k \bC_{n,k}$. Define $\bA_n = \sum_{u\in \bX_n} \bz_l(u)$. Then,
\[
	\bA_n = \sum_{k=1}^{\lfloor \tau_n \rfloor} \bZ_{n,k} \qquad \text{ where } \qquad 
	\bZ_{n,k} = \sum_{c \in \bC_{n,k}} \sum_{u \in \bX_{c}} \bz_l(u).
\]
Using twice the Slivnyak-Mecke Theorem \cite[see e.g.\ Theorem 3.1 in][]{moeller:waagepetersen:04}, 
\begin{align*}
\EE \bZ_{n,k} &= \lfloor \tau_n \rfloor^{-1}\int_{W_n} 	\bz_l(u) \lambda_n(u) \dd u \\
\Var \bZ_{n,k}& = \lfloor \tau_n  \rfloor^{-1}  
\bigg\{ \int_{W_n} 	\bz_l(u) \bz_l(u)^\top \lambda_n(u) \dd u \\
& \qquad 
+ \int_{W_n } \int_{W_n} \bz_l(u)\bz_l(v)^\top\lambda_n(u) \lambda_n (v) \left( g_n(u,v)-1\right) \dd u \dd v
\bigg\}
\end{align*}
where $g_n(u,v)= 1+ \theta_n^{-1} (k*k)(v-u)$. Hence,
\begin{align*}
	\EE \bA_n = \int_{W_n} \bz_l(u) \lambda_n(u) \dd u \qquad  \text{and} \qquad 
	\Var \bA_n = \boldsymbol \Sigma_{l,n}.
\end{align*}
Moreover, by \eqref{eq:unbiased},
\[ \int_{W_n} \bz_l(u) \lambda_n(u) \dd u = \theta_n \int_{W_n} \bz_l(u) \rho_n(u;\bbeta_{l,n}^*) \dd u \] 
so that
$e_n(\bbeta_{l,n}^*)=\bA_n- \EE \bA_n = \sum_{k=1}^{\lfloor \tau_n
  \rfloor} \mathring{\bZ}_{n,k}$ with
$\mathring{\bZ}_{n,k}=\bZ_{n,k}-\EE \bZ_{n,k}$. Since
C\ref{c:variance} is satisfied as shown in Example~\ref{ex:PC} it
follows that the elements of $\Var(\mathring\bZ_{n,k})$ are ${\mathcal
  O}(1)$ whereby $\sum_{k \ge 1} \Var(\mathring\bZ_{n,k})/k^2
<\infty$ (elementwise). Therefore, we can apply Kolmogorov's strong law of large numbers to establish that ${\lfloor \tau_n \rfloor}^{-1} e_n(\bbeta_{l,n}^*)$ tends almost surely to 0. Using the same conditions and C\ref{c:eigen}, we can also apply the Lindeberg-Feller theorem and obtain that as $n\to \infty$, $\boldsymbol \Sigma_{l,n}^{-1/2} e_n(\bbeta_{l,n}^*) \to N(0,\mathbf I_{p_l})$ in distribution.
\end{proof}


\section{Proof of Proposition~\ref{prop:varinvidoni}}

For a multi-index $\alpha \in \NN^p$ with cardinality $|\alpha|=\alpha_1+\dots+\alpha_p$ and a $|\alpha|$ times differentiable function $f:\R^p \to \R$ we define for $u\in \R^p$,
	\begin{equation*}
	\partial^\alpha f(u) = \frac{\partial^{|\alpha|}f(u)}{\partial u_1^{\alpha_1} \dots \partial u_p^{\alpha_p} }.
	\end{equation*}
	 and we also use the notation $u^\alpha= u_1^{\alpha_1} \dots u_p^{\alpha_p}$.  

\begin{proof}
We follow the sketch of the proof of \citet[Lemmas 1-2]{varin:vidoni:05}. Let $\bZ_n$ be an independent copy of $\bX_n$. Let $\ell_n(\bbeta_l,\bZ_n)$ and $e_n(\bbeta_l;\bZ_n)$ be the composite likelihood and its corresponding estimating equation evaluated at $\bZ_n$. And we remind the notation $\ell_n(\bbeta_l)=\ell_n(\bbeta_l;\bX_n)$, $C_n(\bbeta_l)=-\EE \ell_n(\bbeta_l)$, $e_n(\bbeta_l)=e_n(\bbeta_l;\bX_n)$ and $\hat \bbeta_{l,n} = \mathrm{argmax}_{\bbeta_l} \ell_n(\bbeta_l)$. We have
\[
	\EE C_n(\hat \bbeta_{l,n}) = - \EE \left[ \EE \left\{ \ell_n(\hat \bbeta_{l,n} ; \bZ_n) \right\}  \big| \bX_n \right].
\]
Using a first order Taylor expansion of $\ell_n( \cdot ; \bZ_n)$ around $\bbeta_{l,n}^*$ with integral remainder term, we have
\[
	\ell_n(\hat \bbeta_{l,n} ; \bZ_n) = \ell_n(\bbeta_{l,n}^*; \bZ_n) + e_n(\bbeta_{l,n}^*;\bZ_n)^\top \bDelta_n  + \mathbf R_n(\bbeta_{l,n}^*,\hat \bbeta_{l,n}; \bZ_n)
\]
where $\bDelta_n=\hat \bbeta_{l,n}-\bbeta_{l,n}^*$ and where the integral remainder term can be expressed as
\begin{align*}
\mathbf R_n (\bbeta_{l,n}^*,\hat \bbeta_{l,n}; \bZ_n)=& 2 \sum_{\alpha \in \mathbb N^{p_l}, |\alpha|=2} 
\bigg[ \;
\frac{ (\bbeta_{l,n}^*-\hat \bbeta_{l,n})^\alpha}{\alpha!} \\
& \qquad \times 
\int_0^1 (1-t) \partial^\alpha \ell_n\left\{\hat \bbeta_{l,n} + t(\bbeta_{l,n}^*-\hat \bbeta_{l,n}); \bZ_n \right\} \mathrm d t \bigg].
\end{align*}
It is worth noticing that for  any $\bbeta_l \in \R^{p_l}$ and any $\alpha \in \mathbb N^{p_l}$ such that $|\alpha|=2$, $\partial^\alpha \ell_n(\bbeta_l; \bZ_n )$ is a deterministic function of $\bbeta_l$. Therefore $\mathbf R_n (\bbeta_{l,n}^*,\hat \bbeta_{l,n}):=\mathbf R_n (\bbeta_{l,n}^*,\hat \bbeta_{l,n}; \bZ_n)$ does not depend on $\bZ_n$. Using this and the unbiasedness of $e_n(\bbeta_{l,n}^* ; \cdot)$ we have
\[
	\EE \left\{ \ell_n(\hat \bbeta_{l,n};\bZ_n) \big| \bX_n \right\} = \EE \left\{\ell_n(\bbeta_{l,n}^*;\bZ_n) \right\} 
	+ \mathbf R_n(\bbeta_{l,n}^* , \hat \bbeta_{l,n})
\]
Hence,
\begin{equation}\label{eq:ECn}
	\EE \left\{ C_n(\hat \bbeta_{l,n}) \right\} = - \EE \left\{ \ell_n(\bbeta_{l,n}^*)\right\} - \EE \left\{ \mathbf R_n(\bbeta_{l,n}^* , \hat \bbeta_{l,n})\right\}.
\end{equation}
Now, using a first order Taylor expansion of $-\ell_n( \cdot)$ around $\bbeta_{l,n}^*$, we have
\begin{equation}\label{eq:ln2}
	-\ell_n(\hat \bbeta_{l,n}) = -\ell_n(\bbeta_{l,n}^*) - e_n(\bbeta_{l,n}^*)^\top \bDelta_n - \mathbf R_n(\bbeta_{l,n}^*,\hat \bbeta_{l,n}).
\end{equation}
Combining~\eqref{eq:ECn} and the expectation of~\eqref{eq:ln2} we obtain
\begin{align}
\EE \left\{ C_n(\hat \bbeta_{l,n}) \right\} - \EE\left\{ -\ell_n(\hat \bbeta_{l,n})\right\} =& 
\EE \left\{ e_n(\bbeta_{l,n}^*)^\top \bDelta_n\right\}. \label{eq:Cn-ln}	
\end{align}
Since by definition of $\hat \bbeta_{l,n}$, $\bDelta_n = \bS_n^{-1}(\tilde \bbeta_{l,n}) e_n(\bbeta_{l,n}^*)$ where $\tilde \bbeta_{l,n}=\hat \bbeta_{l,n}+ t (\hat \bbeta_{l,n}-\bbeta_{l,n}^*)$ for some $t\in (0,1)$
 we continue with
 \begin{align}
\EE \left\{ C_n(\hat \bbeta_{l,n}) \right\} - \EE\left\{ -\ell_n(\hat \bbeta_{l,n})\right\} =&\EE \left\{ e_n(\bbeta_{l,n}^*)^\top \bS_n(\bbeta_{l,n}^*)^{-1} e_n(\bbeta_{l,n}^*)\right\} \nonumber\\
 &+ \EE \left\{  e_n(\bbeta_{l,n}^*)^\top \mathbf M_n e_n(\bbeta_{l,n}^*)
 \right\}	\nonumber\\
 =& \mathrm{trace} \left\{ \bS_n(\bbeta_{l,n}^*)^{-1} \bSigma_{l,n}\right\} \nonumber\\
 &+ \EE \left\{  e_n(\bbeta_{l,n}^*)^\top \mathbf M_n e_n(\bbeta_{l,n}^*) \right\}
 \end{align}
where $\mathbf M_n = \bS_n(\tilde \bbeta_{l,n})^{-1} -
\bS_n(\bbeta_{l,n}^*)^{-1}$. Since $e_n(\bbeta_{l,n}^*)/\sqrt{\tau_n}$
is bounded in probability by C\ref{c:variance} and $\tau_n\mathbf M_n$
converges in probability since $\bDelta_n$ converges to zero in
probability, we obtain $e_n(\bbeta_{l,n}^*)^\top \mathbf M_n
e_n(\bbeta_{l,n}^*)$ converges in probability. This combined with
uniform integrability \eqref{eq:assumptionUI} implies $\EE \left\{  e_n(\bbeta_{l,n}^*)^\top \mathbf M_n e_n(\bbeta_{l,n}^*) \right\}=o(1)$.
\end{proof}


	\section{Proof of Proposition~\ref{thm:laplace}} \label{sec:thmlaplace}

	\begin{proof}
		We remind the notation $\tau_n = \theta_n |W_n|$.
		\begin{align}
		p_n (\bX_n \mid \mathcal M_l) &= \int_{\R^{p_l}}  p_n(\bX_n,\bbeta_l \mid \mathcal M_l) \dd \bbeta_l \nonumber\\
		& = \int_{\R^{p_l}} p_n(\bX_n; \bbeta_l) p(\bbeta_l \mid \mathcal M_l) \dd \beta_l  \nonumber\\
		&= \int_{{\R^{p_l}}}  p(\bbeta_l \mid \mathcal M_l) \exp\left\{
		\tau_n \, \ell_n(\bbeta_l) /\tau_n  \right\} \dd \bbeta_l. \label{eq:int}
		\end{align}
	Now, we are in the situation where we can apply \citet[Theorem~2]{lapinski:19}, which gives rigorous conditions under which a multivariate Laplace approximation holds. First, condition~C\ref{c:prior} ensures that $p(\bbeta_l \mid \mathcal M_l)$ is regular enough. Second, condition~C\ref{c:eigen} ensures that $\tau_n \ell_n(\bbeta_l)$ has a nonsingular Hessian matrix for any $\bbeta_l$. Third, for $\alpha\in \NN^{p_l}$ with $|\alpha|\le 3$, we have for any $\bbeta_l$ (since $z_{0l}^{\alpha_0}={1}^{\alpha_0}=1$)
		\begin{align*}
		\partial^\alpha \{\tau_n^{-1}\log p_n(\bX_n;\bbeta_l) \} =& -\frac1{|W_n|}\int_{W_n} \prod_{j  \in I_l} z_j^{\alpha_j}(u)\rho(u;\bbeta_l) \dd u. 
		\end{align*}
		Therefore, under condition C\ref{c:boundedcov},
		$
		\partial^\alpha \{\tau_n^{-1} \log p_n(\bX_n;\bbeta_l) \} 
		$
		is uniformly bounded for $\bbeta_l$ in any compact subset of $\R^d$.

The fact that $\int p(\bbeta_l \mid \mathcal M_l) \dd
\bbeta_l=1<\infty$ is sufficient to ensure \citet[condition
(5)]{lapinski:19}. Finally, by the strong consistency of $\hat \bbeta_{l,n} -\bbeta_{l,n}^*$ to 0 and from
condition C\ref{c:compact2}, the last condition to check is formulated as follows: we need to verify  that there exists $n_0\in \mathbb N$ such that for $n\ge n_0$, $\tau_n^{-1} \ell_n(\bbeta_l)$ has a unique maximum in a closed ball $B(\bbeta_{l}^*,\varepsilon) $ for some $\varepsilon>0$ where $\bbeta_{l}^*$ is given by condition C\ref{c:compact2})  and such that
\begin{equation}
	\label{eq:Delta}
	\Delta = \inf_{n\ge n_0, \bbeta_l \in \R^{p_l} \setminus B(\bbeta_l^*,\varepsilon)} \; \tau_n^{-1} 
	\left\{  
\ell_n(\hat \bbeta_{l,n}) - \ell_n (\bbeta_l)
	\right\} >0.
\end{equation}

By condition C\ref{c:eigen}, for any $n\ge 1$, $\ell_n$
(and thus $\tau_n^{-1} \ell_n$) has indeed a unique maximum. Letting $\varepsilon>0$ there exists $n_0\in \NN$ such that almost surely $\hat \bbeta_{l,n}\in B(\bbeta_l^*,\varepsilon/2)$. Consider $\Delta$ with such $n_0$ and $\varepsilon$. Note that for large $n$ and any $\bbeta_l \in \R^{p_l}\setminus B(\bbeta_l^*,\varepsilon)$, $\|\hat \bbeta_{l,n}-\bbeta_l\|\ge \varepsilon/2$ almost surely.

Using a first order Taylor expansion around $\hat \bbeta_{l,n}$ and using the fact that $e_n(\hat \bbeta_{l,n})=0$, we have
\[
\tau_n^{-1}  \left\{  
 \ell_n (\bbeta_l)- \ell_n(\hat \bbeta_{l,n}) 	\right\}=  R_n (\hat \bbeta_{l,n}, \bbeta_{l,n})
\]
where the remainder term $ R_n (\hat \bbeta_{l,n}, \bbeta_{l,n})$ is 
\begin{align*}
& \frac{2}{\tau_n} \sum_{\alpha \in \mathbb N^{p_l}, |\alpha|=2} 
\bigg[ \;
\frac{ (\bbeta_{l,n}-\hat \bbeta_{l,n})^\alpha}{\alpha!} \\
& \qquad \times 
\int_0^1 (1-t) \partial^\alpha \ell_n\left\{\hat \bbeta_{l,n} +  t(\bbeta_{l,n}-\hat   \bbeta_{l,n});  \bX_n )\right\}\dd t \bigg]\\
=& \frac{-2}{|W_n|} \sum_{\alpha \in \mathbb N^{p_l}, |\alpha|=2} \bigg[ \;\frac{ (\bbeta_{l,n}-\hat \bbeta_{l,n})^\alpha}{\alpha!} \\
& \qquad \times \int_0^1 (1-t) \int_{W_n} \bz^{\alpha}(u)  \rho(u;\hat \bbeta_{l,n} +  t(\bbeta_{l,n}-\hat   \bbeta_{l,n})) \dd u \dd t\bigg]\\
=& \frac{-2}{|W_n|}  \int_{W_n} \sum_{\alpha \in \mathbb N^{p_l},
   |\alpha|=2}  \;\frac{ (\bbeta_{l,n}-\hat
   \bbeta_{l,n})^\alpha}{\alpha!}  \bz^{\alpha}(u) \\ & \times \exp\{\hat
   \bbeta_{l,n}\bz_l(u)\} \frac{\exp\{(\bbeta_{l,n}-\hat   \bbeta_{l,n})^\top\bz_l(u)\}- (\bbeta_{l,n}-\hat   \bbeta_{l,n})^\top\bz_l(u)-1}{\{(\bbeta_{l,n}-\hat   \bbeta_{l,n})^\top\bz_l(u)\}^2}\dd u\\  
=& \frac{-1}{|W_n|}  \int_{W_n}  \exp\{\hat  \bbeta_{l,n}^\top \bz_l(u)\} \,
f \left\{ (\bbeta_{l,n}-\hat   \bbeta_{l,n})^\top\bz_l(u)\right\}
\dd u
\end{align*}
where $f(t)=\exp(t)-1-t$. Consider the set $B_n$ given by C\ref{c:eigen}. 
Since for large $n$, $\|\hat \bbeta_{l,n}-\bbeta_{l}^*\| \ge \varepsilon/2$ and since $f$ is non negative and has a unique minimum at zero, we have by condition C\ref{c:eigen} that
\[
   \inf_{u \in B_n} \exp\{\hat  \bbeta_{l,n}^\top \bz_l(u)\} > \exp(- \| \hat  \bbeta_{l,n}\| c) \text{ and }	\inf_{u\in B_n} f\left\{ (\bbeta_{l,n}-\hat   \bbeta_{l,n})^\top\bz_l(u)\right\} \ge \delta>0   
\]
for some $\delta>0$. Therefore, 
\begin{align*}
	 R_n (\hat \bbeta_{l,n}, \bbeta_{l,n}) \leq&
\frac{-1}{|W_n|}  \int_{B_n}  \exp\{\hat  \bbeta_{l,n}^\top \bz_l(u)\} \,
f \left\{ (\bbeta_{l,n}-\hat   \bbeta_{l,n})^\top\bz_l(u)\right\}
\dd u \\
&\leq  - \delta \,\frac{|B_n|}{|W_n|}  \, \exp(- \| \hat  \bbeta_{l,n}\| c).
\end{align*}
Again, by definition of $B_n$, the strong consistency of $\hat
\bbeta_{l,n}$ and condition C\ref{c:boundedcov}, we conclude that for
large $n$ there exists $\delta^\prime>0$ such that $ R_n (\hat
\bbeta_{l,n}, \bbeta_{l,n}) \leq -\delta^\prime <0 $ whereby we deduce
that $\Delta>0$.

The previous statements allow us to conclude by \citet[Theorem~2]{lapinski:19} that
\begin{align*} & p_n(\bX_n \mid \mathcal M_l) \\ = & \frac{p ( \hat \bbeta_{l,n} \mid \mathcal M_l)}{\det \{\tau_n^{-1} \bS_n (\hat \bbeta_{l,n}) \}^{1/2}}  \exp \{ \ell_n(\hat \bbeta_{l,n}) \} \left ( \frac{2 \pi}{\tau_n} \right)^{p_l/2} \!\!\!\! + \exp \{ \ell_n(\hat \bbeta_{l,n}) \} \left ( \frac{2 \pi}{\tau_n} \right)^{p_l/2} \!\!\!\! {\mathcal O}(\tau_n^{-1/2}).
\end{align*}
This gives
\begin{multline*} \log p_n(\bX_n \mid \mathcal M_l) - \log \frac{p ( \hat \bbeta_{l,n} \mid \mathcal M_l)}{\det \{\tau_n^{-1} \bS_n (\hat \bbeta_{l,n}) \}^{1/2}}  - \ell_n(\hat \bbeta_{l,n}) - \frac{p_l}{2} \log \frac{2 \pi}{\tau_n}  =\\ \log \left [1 + \frac{ \det \{\tau_n^{-1} \bS_n (\hat \bbeta_{l,n}) \}^{1/2}}{p ( \hat \bbeta_{l,n} \mid \mathcal M_l)} {\mathcal O}(\tau_n^{-1/2}) \right ].\end{multline*}
Proposition~\ref{thm:laplace} now follows since
\[ \log \left [ 1 + \frac{ \det \{\tau_n^{-1} \bS_n (\hat \bbeta_{l,n}) \}^{1/2}}{p ( \hat \bbeta_{l,n} \mid \mathcal M_l)} {\mathcal O}(\tau_n^{-1/2}) \right ] = \log \{ 1+ 
  {\mathcal O}(\tau_n^{-1/2} )\}= {\mathcal O}(\tau_n^{-1/2} ) \]
and using condition~C\ref{c:mun}.
\end{proof}

\section{Modified version of Theorem~2 in \cite{waagepetersen:guan:09}}\label{app:generalresult}\label{sec:wg09}

 Consider a sequence of estimating functions $u_n:\R^p \rightarrow
\R^p$, $n \ge 1$ whose distribution is determined by some underlying probability measure generating the data at hand. For a matrix $\bA=[a_{ij}]$, $\|\bA\|_M= \max_{ij}|a_{ij}|$, and
we let $\mathbf J_n(\boldsymbol\ta)= - \frac{\dd}{{\dd \ta}} u_n(\boldsymbol\ta)$, assuming that $u_n$ is differentiable.

\begin{theorem}\label{thm:general}
Assume that there exists a sequence of invertible symmetric matrices
$V_n$ and a sequence of parameter values $\boldsymbol\ta_n^* \in \R^p$ such that
\begin{enumerate}
\renewcommand{\theenumi}{\arabic{enumi}}
\renewcommand{\labelenumi}{G\theenumi}
\item \label{cond:tozero} $\|{\mathbf V}_n^{-1}\| \rightarrow 0$.
\item \label{cond:posdef} There exists an $l>0$ so that ${\P}( l_n <
  l)$ tends to zero where
  $$
  l_n= \inf_{\| {\boldsymbol\phi}\|=1} 
   {\boldsymbol\phi^\top} {\mathbf V}_n^{-1} {\mathbf{J}}_n({\boldsymbol\ta}_n^*) {\mathbf V}_n^{-1}
  {\boldsymbol\phi}.
  $$
\item \label{cond:contgeneral}
For any $d>0$,
\[  \sup_{\| {\mathbf V}_n({\boldsymbol\ta}- {\boldsymbol\ta}^*_n)   \| \le d} \| {\mathbf V}_n^{-1} \{ {\mathbf J}_n({\boldsymbol\ta})-{\mathbf J}_n({\boldsymbol\ta}^*_n) \}
{\mathbf V}_n^{-1} \|_M =
 \gm_{nd} \rightarrow 0\]
in probability under $P$.

\item \label{cond:unormal}
The sequence $u_n({\boldsymbol\ta}^*_n) {\mathbf V}_n^{-1}$ is bounded in probability (i.e.\
for each $\epsilon>0$ there exists a $d$ so that $\P(\|{ {\mathbf V}_n^{-1} u_n({\boldsymbol\ta}^*_n)}
\| >d) \le \epsilon$ for $n$ sufficiently large).
 \end{enumerate}
Then for each $\epsilon>0$, there exists a $d>0$ such that
\begin{equation}\label{eq:consistent} {\P} \left\{ \exists \tilde{{\boldsymbol\ta}}_n :
  u_{n}(\tilde{{\boldsymbol\ta}}_n)=0 \text{ and }
  \| {{\mathbf V}_n(\tilde{{\boldsymbol\ta}}_n - {\boldsymbol\ta}^*_n) } \| < d  \right\} > 1- \epsilon
\end{equation}
whenever $n$ is sufficiently large.
\end{theorem}
\begin{proof}
The event
\[ \{ \exists  \tilde{{\boldsymbol\ta}}_n : u_n(\tilde{{\boldsymbol\ta}}_n)=0 \text{ and }
  \| {{\mathbf V}_n(\tilde{{\boldsymbol\ta}}_n - {\boldsymbol\ta}^*_n) }\| < d \}
\]
occurs if 
${ {\boldsymbol\phi}^\top {\mathbf V}_n^{-1}  u_n({\boldsymbol\ta}^*_n+  {\mathbf V}_n^{-1} {\boldsymbol\phi} )}
 <0$ for all
${\boldsymbol\phi}$ with $\| {\boldsymbol\phi} \|=d$ since this implies ${u_n({\boldsymbol\ta}^*_n+  {\mathbf V}_n^{-1} {\boldsymbol\phi} )}$ for some $\| {\boldsymbol\phi}\| < d$ \cite[Lemma~2
in][]{aitchison:silvey:58}. Hence we need to show that there is a
$d$ such that
\begin{equation*}\label{eq:limsupzero} 
\P \left\{ \sup_{\| {\boldsymbol\phi} \| =d}
  { {\boldsymbol\phi}^\top {\mathbf V}_n^{-1}  u_n({\boldsymbol\ta}^*_n+  {\mathbf V}_n^{-1} {\boldsymbol\phi} )}  \ge 0
  \right\}
   \le \epsilon
\end{equation*}
for sufficiently large $n$. To this end we write
\[ { {\boldsymbol\phi}^\top {\mathbf V}_n^{-1}  u_n({\boldsymbol\ta}^*_n+  {\mathbf V}_n^{-1} {\boldsymbol\phi} )}
 = { {\boldsymbol\phi}^\top {\mathbf V}_n^{-1}  u_n({\boldsymbol\ta}^*_n )}
 - {{\boldsymbol\phi}^\top} \int_0^1 {\mathbf V}_n^{-1} {\mathbf J}_n({\boldsymbol\ta}_n(t)) {\mathbf V}_n^{-1} \dd t {{\boldsymbol\phi}}
\] where ${\boldsymbol\ta}(t)={\boldsymbol\ta}^*_n+t { {\mathbf V}_n^{-1}{\boldsymbol\phi}}$. Then
\begin{align*}
& {\P} \left\{ \sup_{\| {\boldsymbol\phi} \| =d} 
{ {\boldsymbol\phi}^\top {\mathbf V}_n^{-1}  u_n({\boldsymbol\ta}^*_n+  {\mathbf V}_n^{-1} {\boldsymbol\phi} )}
\ge 0 \right\} \leq  \\ 
& {\P} \left\{ \sup_{\| {\boldsymbol\phi} \| =d} 
{ {\boldsymbol\phi}^\top {\mathbf V}_n^{-1}  u_n({\boldsymbol\ta}^*_n )}
  \ge \inf_{\| {\boldsymbol\phi} \|
  =d} {{\boldsymbol\phi}^\top} \int_0^1 {\mathbf V}_n^{-1}
{\mathbf J}_n({\boldsymbol\ta}_n(t)){\mathbf V}_n^{-1} \dd t \, {{\boldsymbol\phi}} \right\}  \le \\
& {\P} \left[ \| {{\mathbf V}^{-1}_n u_n({\boldsymbol\ta}^*_n}) \| \ge d  \inf_{\|
  {\boldsymbol\phi} \| =1}  \left\{ {{\boldsymbol\phi}^\top} {\mathbf V}_n^{-1}{\mathbf J}_n({\boldsymbol\ta}^*_n){\mathbf V}_n^{-1} {{\boldsymbol\phi}} \right\} - d
 p \gm_{nd}] \right] \le \\ & {\P}\left( \| {{\mathbf V}_n^{-1}u_n({\boldsymbol\ta}^*_n)}\|  \ge d l_n /2 \right) + {\P}\left(
p \gm_{nd} >  l_n /2\right).
\end{align*}
The first term can be made arbitrarily small by picking a
sufficiently large $d$ and letting $n$ tend to infinity. The second
term converges to zero as $n$ tends to infinity.
\end{proof}

\bibliographystyle{plainnat}
\bibliography{bic}
		
\end{document}